\documentclass[10pt]{article}

\usepackage{amssymb, amsmath}
\usepackage{amsthm}
\usepackage{graphics}
\usepackage{amsfonts}
\usepackage{mathrsfs}
\usepackage{amscd}
\usepackage{comment}
\usepackage{color}
\usepackage{url}
\usepackage[normalem]{ulem}

\newcommand{\bN} { {\mathbb{N}}}

\newcommand{\bQ} { {\mathbb{Q}}}
\newcommand{\bZ} { {\mathbb{Z}}}

\newcommand{\ok} { {\overline{K}}}

\newcommand{\disp}{\operatorname{disp}}
\newcommand{\qdisp}{\operatorname{qdisp}}

\newcommand{\cres}{\operatorname{cres}}
\newcommand{\dres}{\operatorname{dres}}
\newcommand{\qres}{\operatorname{qres}}

\theoremstyle{definition}
\newtheorem{theorem}{Theorem}[section]
\newtheorem{prop}[theorem]{Proposition}

\newtheorem{lemma}[theorem]{Lemma}

\newtheorem{remark}[theorem]{Remark}
\newtheorem{define}[theorem]{Definition}
\newtheorem{example}[theorem]{Example}

\begin{document}
\title{Residues and Telescopers for
Rational Functions\thanks{This work is supported by the NSF grant CCF-1017217.}}
\author{Shaoshi Chen and Michael F.\ Singer\footnote{Email Addresses:~{\tt schen21@ncsu.edu}~(S.\ Chen),
{\tt singer@math.ncsu.edu}~(M.\ F.\ Singer).}\, \footnote{Corresponding Author: Michael F. Singer, Department of Mathematics,
North Carolina State University, Box 8205, Raleigh, NC 27695-8205,
{\tt singer@math.ncsu.edu}, Telephone:  919 515 2671, Fax:  919 515 3798} \\\\
Department of Mathematics, \\North Carolina
State University,\\ Raleigh, NC 27695-8205, USA}
%

\date{\today}
\maketitle
\begin{center} {\em Dedicated to the memories of Philippe Flajolet and Herbert S.~Wilf}
\end{center}

\vspace{.1in}

\begin{abstract}
We give necessary and sufficient conditions for the existence of telescopers
for rational functions of two variables in the continuous,  discrete and~$q$-discrete
settings and characterize which operators can occur as telescopers.
Using this latter characterization, we reprove results of Furstenberg and
Zeilberger concerning diagonals of power series representing rational functions.
The key concept behind these considerations is a generalization of the notion of
residue in the continuous case to an analogous concept in the  discrete and~$q$-discrete cases.
\end{abstract}

\begin{center}
\begin{minipage}[t]{10.5cm}
\end{minipage}
\end{center}

\section{Introduction}\label{SECT:intro}
Residues have played a ubiquitous and important role in mathematics and
their use in combinatorics has had a lasting impact~(e.g., \cite{Flajolet_Sedgewick}).
In this paper we will show how the notion of residue and its generalizations lead
to new results and  a recasting of known results concerning telescopers in the continuous,
 discrete and~$q$-discrete cases.

As an introduction to our point of view and our results,
let us consider the  problem of finding a differential
telescoper for a rational function of two variables.
Let~$k$ be a field of characteristic zero,~$k(t,x)$ the
field of rational functions of two variables
and~$D_t = \partial/\partial_t$ and~$D_x=\partial/\partial_x$
the usual derivations with respect to~$t$ and~$x$,  respectively. Given~$f \in k(t,x)$,
we wish to find a nonzero operator  $L \in k(t)\langle D_t\rangle$, the
ring of linear differential operators in~$D_t$ with coefficients
in $k(t)$, and an element~$g \in k(t,x)$ such that~$L(f) = D_x(g)$.
We may consider~$f$ as an element of~$\overline{K}(x)$ where~$\overline{K}$ is
the algebraic closure of~$K = k(t)$.  As such, we may write
\begin{equation} \label{eqn0}
  f = p + \sum_{i=1}^m \sum_{j=1}^{n_i}\frac{\alpha_{i,j}}{(x-\beta_i)^j},
\end{equation}
where $p\in {K}[x]$, the $\beta_i$ are the roots in $\overline{K}$ of
the denominator of $f$ and the $\alpha_{i,j}$ are in~$\overline{K}$.
Note that the element $\alpha_{i,1}$ is the usual \emph{residue} of~$f$ at~$\beta_i$.
Using Hermite reduction (\cite[p.~39]{BronsteinBook} {or Section~\ref{SUBSECT:cres} below}),
one sees that a rational function $h \in K(x)$ is of the form $h = D_x(g)$ for some $g \in
K(x)$ if and only if all residues of $h$ are zero.  Therefore to find a
telescoper for $f$ it is enough to find a nonzero operator $L \in K\langle
D_t \rangle$ such that $L(f)$ has only zero residues.  For example assume that
$f$ has only simple poles, i.e., $f = \frac{a}{b}, a,b \in K[x]$, $\deg_xa <
\deg_x b$ and $b$ squarefree.  We then know that the Rothstein-Trager resultant
\cite{Trager1976, Rothstein1977}
\[
  R := {\rm resultant}_x(a-zD_x(b), b)\in K[z]
\]
is a polynomial whose roots are the residues at the poles of~$f$.  Given a
squarefree polynomial in $K[z]=k(t)[z]$, differentiation with respect to $t$ and
elimination allow one to construct a nonzero linear differential operator $L \in
k(t)\langle D_t \rangle$ such that $L$ annihilates the roots of this polynomial.
Applying $L$ to each term of \eqref{eqn0} one sees that $L(f)$ has zero residues
at each of its poles. Applying Hermite reduction to $L(f)$ allows us to find a
$g$ such that $L(f) = D_x(g)$.

The main idea in the method described above is that nonzero residues are the
obstruction to being the derivative of a rational function and one constructs a
linear operator to remove this obstruction. This idea is the basis of  results in
\cite{CKS2012} where it is shown that the problem of finding differential telescopers
for rational functions in~$m$ variables is equivalent to the problem of finding telescopers for
algebraic functions in~$m-1$ variables and where a new algorithm for finding
telescopers for algebraic functions in two variables is given.

 For a precise problem description, let~$k(t,x)$ be as above and~$D_t$ and~$D_x$ be the
derivations defined above.  We define shift operators~$S_t$ and~$S_x$
as
\[S_t(f(t,x)) = f(t+1,x) \quad  \text{and} \quad S_x(f(t,x)) = f(t, x+1)\]
and~$q$-shift operators (for~$q\in k$ not a root of unity)~$Q_t$ and~$Q_x$ as
\[Q_t(f(t,x) = f(qt,x)\quad  \text{and} \quad Q_x(f(t,x)) = f(t,qx)).\]
Let~$\Delta_x$ and~$\Delta_{q, x}$ denote the difference
and~$q$-difference operators~$S_x-1$ and~$Q_x-1$, respectively.  In this paper,
we give a solution to the following problem
\begin{center}
\begin{minipage}[t]{10cm}
\underline{Existence Problem for Telescopers.}\,\,
{\it For any~$\partial_t\in \{D_t, S_t, Q_t\}$ and~$\partial_x\in \{D_x, \Delta_x, \Delta_{q,x}\}$
find necessary and sufficient conditions on elements~$f \in k(t,x)$ that guarantee the existence
of a nonzero linear operator~$L(t,\partial_t)$ in~$\partial_t$ with coefficients in~$k(t)$
(a {\emph{telescoper}}) and an element~$g \in k(t,x)$ (a {\emph{certificate}}) such that
\[L(t,\partial_t)(f) = \partial_x(g).\]}
\end{minipage}
\end{center}
\vspace{-0.3cm}

As we have shown above, when~$\partial_t = D_t$
and~$\partial_x = D_x$, a telescoper and certificate exist for
any~$f\in k(t,x)$. This is not necessarily true in the other cases.
In the case when~$\partial_t = S_t$ and~$\partial_x = \Delta_x$,
Abramov and Le~\cite{AbramovLe2002} showed that there is no telescoper
for the rational function~$1/(t^2+x^2)$ and  presented a necessary and sufficient
condition for the existence of telescopers. Later, Abramov gave a general criterion
for the existence of telescopers for hypergeometric terms~\cite{Abramov2003}.
The~$q$-analogs were achieved in the works by Le~\cite{Le2001} and
by Chen et al. \cite{Chen2005}. Our approach in this paper represents
a unified way of solving the Existence Problem for Telescopers (for rational functions)
in these and the remaining cases.
In particular, we will first identify in each case the appropriate notion
of ``residues''  which will be elements of $\overline{k(t)}$, the algebraic
closure of~$k(t)$. We will show that for any~$f \in k(t,x)$
and~$\partial_x \in \{D_x, \Delta_x, \Delta_{q,x}\}$, there exists a~$g \in k(t,x)$
such that~$f = \partial_x(g)$ if and only if all the ``residues'' vanish.
We will then show that to find a telescoper, it is necessary and sufficient to
find an operator~$L(t,\partial_t)$ that annihilates all of the residues.

This necessary and sufficient condition has several applications.
For example, our results reduce the Existence Problem for Telescopers to the problem
of finding necessary and sufficient conditions that guarantee the existence
of operators that annihilate algebraic functions and we present a solution to
this latter problem.  Our approach also gives
termination criteria for the Zeilberger method~\cite{Almkvist1990, Zeilberger1990, Zeilberger1991} and
also a strategy for finding telescopers and certificates, which has been successfully used
in the continuous case in~\cite{CKS2012}.
In addition, these criteria  together with the
results in~\cite{Hardouin2008, Schneider2010} can be used to determine
if indefinite sums and integrals satisfy (possibly nonlinear) differential equations (see Example~\ref{EX:transcendental}).

The rest of the paper is organized as follows. In Section~\ref{SECT:residues} we
define the notions of  residues relevant to the discrete and~$q$-discrete
cases and show that for any~$f \in k(t,x)$ and~$\partial_x \in \{D_x, \Delta_x, \Delta_{q,x}\}$,
there exists a~$g \in k(t,x)$ such that~$f = \partial_x g$ if and only if all the residues vanish.
In Section~\ref{SECT:algfuns} we characterize those algebraic functions in~$\overline{k(t)}$ for
which there exist annihilating linear operators~$L(t,S_t)$ or $L(t,Q_t)$ as well as prove some
ancillary results useful in succeeding sections. In Section~\ref{SECT:telescoper}, we solve the
Existence Problem for Telescopers as well as characterize when a linear operator is a telescoper.
Using this latter characterization, we can give a proof, using our approach, of the theorem
of Furstenberg~\cite{Furstenberg1967} stating that the diagonal of a rational power series in
two variables is an algebraic function.  We also discuss a recent example
of Ekhad and Zeilberger~\cite{EZ2011} in the context of the results of this paper.
The final Appendix contains proofs of the characterizations stated in Section~\ref{SECT:algfuns}.

\section{Residues}\label{SECT:residues}
Let~$K$ be a field of characteristic zero and~$K(x)$ be the field of rational functions in~$x$ over~$K$.
Let~$\overline{K}$ denote the algebraic closure of~$K$. Let~$q\in K$ be such that~$q^i\neq 1$ for
any nonzero~$i\in \bZ$,
i.e., $q$ is not a root of unity.  As in the Introduction, we define the derivation~$D_x$, shift operator~$S_x$,
and~$q$-shift operator~$Q_x$ on~$K(x)$, respectively, as
\[D_x(f(x))=\frac{d(f(x))}{dx},  \quad S_x(f(x))= f(x+1), \quad \text{and}\quad Q_x(f(x))=f(qx) \]
for all~$f\in K(x)$. Let~$\Delta_x$ and~$\Delta_{q, x}$ denote the difference and~$q$-difference
operators~$S_x-1$ and~$Q_x-1$, respectively.
A rational function~$f\in K(x)$ is said to be
\emph{rational integrable} (resp.\ \emph{summable, $q$-summable}) in~$K(x)$ if
there exists~$g\in K(x)$
such that~$f=D_x(g)$ (resp.\ $f=\Delta_x(g)$, $f = \Delta_{q, x}(g)$). This section is motivated
by the well known result (Proposition~\ref{PROP:ratint} below)  that characterizes rational integrability
in terms of vanishing residues. In the remainder of this section we describe other types
of ``residues'' and how they can be used to give necessary and sufficient conditions for summability
and~$q$-summability.

\subsection{Continuous residues}\label{SUBSECT:cres}
Let~$f=a/b\in K(x)$ with~$a, b\in K[x]$ and~$\gcd(a, b)=1$. Then~$f$ can be uniquely
written in its partial fraction decomposition
\begin{equation}\label{EQ:cparfrac}
f = p + \sum_{i=1}^m \sum_{j=1}^{n_i} \frac{\alpha_{i,j}}{(x-\beta_i)^j},
\end{equation}
where~$p\in K[x]$, { $m, n_i\in \bN$, $\alpha_{i,j}, \beta_i\in \overline{K}$, and~$\beta_j$'s are roots of~$b$.
From any of the usual proofs of partial fraction decompositions, one sees that all the~$\alpha_{i, j}$'s are
in~$K(\beta_1, \ldots, \beta_m)$.
\begin{define}[Continuous residue]\label{DEF:cres}
Let~$f\in K(x)$ be of the form~\eqref{EQ:cparfrac}.
The value~$\alpha_{i,1}\in \ok$ is called the \emph{continuous residue}
of~$f$ at~$\beta_i$ (with respect to~$x$), denoted by~$\cres_x(f, \beta_i)$.
\end{define}
\noindent
Note that the continuous residue is just the usual residue  in complex analysis.
We will define other kinds of residues below but  when we refer to a residue
without further modification, we shall mean the continuous residue.
Although the following is well known (see~\cite[Proposition 2.1]{VanDerPut2001})
we include it since this result is the motivation and model for the considerations that follow.
\begin{prop}\label{PROP:ratint}
Let~$f=a/b\in K(x)$ be such that~$a, b\in K[x]$ and~$\gcd(a, b)=1$.
Then~$f$ is rational integrable in~$K(x)$ if and only if the residue~$\cres_x(f, \beta)$
is zero for any root~$\beta\in \ok$ of~$b$.
\end{prop}
\begin{proof}
Suppose that~$f$ is rational integrable in~$K(x)$, i.e.,\ $f= D_x(g)$ for some~$g$ in~$K(x)$.
Writing~$g$ in its partial fraction decomposition and differentiating each term,
one sees that all the residues of~$D_x(g)$ are~$0$. Conversely, if all residues
of~$f$ at its poles are zero, then~$f$ can be written as
\[f = p + \sum_{i=1}^m \sum_{j=2}^{n_i} \frac{\alpha_{i,j}}{(x-\beta_i)^j},\]
where~$p\in K[x]$, $\alpha_{i,j}, \beta_i\in \overline{K}$, and~$\beta_j$'s are roots of~$b$.
Note that any polynomial is rational integrable in~$K(x)$, and for all~$i, j$ with~$1\leq i\leq m$ and~$2\leq j\leq n_i$,
\[\frac{\alpha_{i,j}}{(x-\beta_i)^j} = D_x\left(\frac{(1-j)^{-1}\alpha_{i, j}}{(x-\beta_i)^{j-1}}\right).\]
Then~$f=D_x(g)$, where~$g$ is of the form
\[g = \tilde{p} + \sum_{i=1}^m \sum_{j=2}^{n_i} \frac{(1-j)^{-1}\alpha_{i, j}}{(x-\beta_i)^{j-1}} \quad \text{for some~$\tilde{p}\in K[x]$.}\]
For each irreducible factor $p$ of~$b$, the sum in~$g$ is a symmetric function of
those~$\beta_i$'s that are roots of~$p$.  From this one concludes that $g$ lies in~$K(x)$. Thus,~$f$ is rational
integrable in~$K(x)$.
\end{proof}

\subsection{Discrete residues}\label{SUBSECT:dres}
Given a rational function, Matusevich~\cite{Matusevich2000}
found a necessary and sufficient condition
for its rational summability. Moreover, one can algorithmically
decide whether a rational function is
rational summable or not using methods
in~\cite{Abramov1975, Abramov1989,
AbramovYu1991, Abramov1995, Abramov1995b, Paule1995b, Pirastu1995a, Pirastu1995b}.
Here, we present a rational summability criterion
via a discrete analogue of residues. To this end, we first recall some terminology
from~\cite{Abramov1975, Paule1995b}  and~\cite[Chapter 2]{vdPutSinger1997}.

For an element~$\alpha \in \ok$, we call the subset~$\alpha + \bZ$ the~\emph{$\bZ$-orbit} of~$\alpha$
in~$\ok$, denoted by~$[\alpha]$. For a polynomial~$b\in K[x]\setminus K$, the value
\[\max\{i\in \bZ \mid \text{$\exists \, \alpha, \beta \in \ok$
such that~$i=\alpha-\beta$ and~$b(\alpha)=b(\beta)=0$}\}\]
is called the~\emph{dispersion} of~$b$ with respect to~$x$,
denoted by~$\disp_x(b)$. The polynomial~$b$ is said to
be~\emph{shift-free} with respect to~$x$ if~$\disp_x(b)=0$.
Let~$f=a/b\in K(x)$ be such that~$a, b\in K[x]$ and~$\gcd(a, b)=1$.
Over the field~$\ok$, $f$ can be decomposed into the form
\begin{equation}\label{EQ:dparfrac}
f =  p + \sum_{i=1}^m \sum_{j=1}^{n_i} \sum_{\ell=0}^{d_{i,j}} \frac{\alpha_{i,j,\ell}}{(x-(\beta_i+\ell))^{j}},
\end{equation}
where~$p\in K[x]$,  $m, n_i, d_{i,j}\in \bN$, $\alpha_{i, j, \ell}, \beta_i\in \ok$, and~$\beta_i$'s are
in distinct $\bZ$-orbits.

\begin{define}[Discrete residue]\label{DEF:dres}
Let~$f\in K(x)$ be of the form~\eqref{EQ:dparfrac}.
The sum~$\sum_{\ell=0}^{d_{i,j}}\alpha_{i,j, \ell}$
is called the \emph{discrete residue} of~$f$ at
the $\bZ$-orbit $[\beta_i]$ of multiplicity~$j$ (with respect to~$x$),
denoted by~$\dres_x(f, [\beta_i], j)$.
\end{define}

\begin{lemma}\label{LM:dres}
Let~$f=\sum_{\ell=0}^d \alpha_{\ell}/(x-(\beta+\ell))^s$ be
such that~$d, s\in \bN$ and~$\alpha_{\ell}, \beta\in \ok$.
Then~$f$ is rational summable in~$\ok(x)$ if and only if
the sum~$\sum_{\ell=0}^d \alpha_{\ell}$ is zero that is, if and only if~$\dres_x(f, [\beta], s)=0$.
\end{lemma}
\begin{proof}
Suppose that the sum~$\sum_{\ell=0}^d \alpha_{\ell}$ is zero.
We show that~$f$ is rational summable in~$\ok(x)$.
To this end, we proceed by induction on~$d$.
In the base case when~$d=0$, $f$ is clearly rational summable in~$\ok(x)$ since~$f=0$.
Suppose that the assertion holds for~$d=m$ with~$m\geq 0$. Note that
\begin{align*}
 \frac{\alpha_{m+1}}{(x-(\beta+m+1))^s} =
    \Delta_x \left(-\frac{\alpha_{m+1}}{(x-(\beta+m+1))^s}\right) +
    \frac{\alpha_{m+1}}{(x-(\beta+m))^s}.
\end{align*}
This implies that
\[\sum_{\ell=0}^{m+1} \frac{\alpha_{\ell}}{(x-(\beta+\ell))^s} =
\Delta_x \left(-\frac{\alpha_{m+1}}{x-(\beta+m+1)^s}\right) +
\sum_{\ell=0}^{m} \frac{\tilde{\alpha}_{\ell}}{(x-(\beta+\ell))^s},\]
where~$\tilde{\alpha}_{\ell} = \alpha_{\ell}$ if~$0\leq \ell \leq m-1$
and~$\tilde{\alpha}_{m} = \alpha_{m+1} + \alpha_{m}$. By definition, the
sum~$\sum_{\ell=0}^m \tilde{\alpha}_{\ell}$ is still zero.
 The induction hypothesis then implies that there exists~$\tilde{g}\in \ok(x)$ such that
\[\sum_{\ell=0}^{m} \frac{\tilde{\alpha}_{\ell}}{(x-(\beta+\ell))^s} = \Delta_x(\tilde{g}).\]
So~$f=\Delta_x(g)$ with~$g=\tilde{g} - \alpha_{m+1}/(x-(\beta+m+1))^s\in \ok(x)$.
For the opposite implication, we assume to the contrary that
the sum~$\sum_{\ell=0}^d \alpha_\ell$
is nonzero. Without loss of generality, we can assume that~$\alpha_0 \neq 0$.
Write~$\alpha_0 = \bar{\alpha}_0 + \tilde{\alpha}_0$ such that~$\tilde{\alpha}_0 +
\sum_{\ell=1}^{d} \alpha_\ell =0$. Since~$\sum_{\ell=0}^d \alpha_\ell\neq 0$, $\bar{\alpha}_0\neq 0$.
By the assertion shown above, there exists~$\tilde{g}\in \ok(x)$ such that
\[f = \frac{\bar{\alpha}_0}{(x-\beta)^s} + \Delta_x(\tilde{g}).\]
Since~$\disp_x((x-\beta)^s)=0$ and~$\bar{\alpha}_0\neq 0$,
${\bar{\alpha}_0}/{(x-\beta)^s}$ is not rational
summable by~\cite[Lemma 3]{Matusevich2000} or~\cite[Lemma 6.3]{Hardouin2008}.
Then~$f$ is not rational summable in~$\ok(x)$.
This completes the proof.
\end{proof}
\begin{prop}\label{PROP:ratsum}
Let~$f=a/b\in K(x)$ be such that~$a, b\in K[x]$ and~$\gcd(a, b)=1$.
Then~$f$ is rational summable in~$K(x)$ if and only if the  discrete residue $\dres_x(f, [\beta], j)$
is zero for any $\bZ$-orbit~$[\beta]$ with~$b(\beta)=0$ of any multiplicity~$j\in \bN$.
\end{prop}
\begin{proof}
Let~$f\in K(x)$ be decomposed into the form~\eqref{EQ:dparfrac}.
If the discrete residue of~$f$
at any~$\bZ$-orbit of any multiplicity is zero, then Lemma~\ref{LM:dres} implies that
for all~$i, j$ with~$1\leq i \leq m$
and~$1\leq j \leq n_i$, the sum
\[\sum_{\ell=0}^{d_{i,j}} \frac{\alpha_{i,j,\ell}}{(x-(\beta_i+\ell))^{j}} = \Delta_x(g_{i,j})\quad
\text{for some~$g_{i, j}\in \ok(x)$.}\]
Since any polynomial is rational summable, there exists~$\tilde{p}\in K[x]$
such that~$p = \Delta_x(\tilde{p})$.
So~$f = \Delta_x(\tilde{p} + g)$, where~$g=\sum_{i=1}^m\sum_{j=1}^{n_i} g_{i,j}$.
Arguing as in Proposition~\ref{PROP:ratint}, one sees that
for each irreducible factor $p$ of~$b$, the sum in~$f$ is a symmetric function
of those~$\beta_i$'s that are roots of~$p$.  From this one concludes that the
sum is in~$K(x)$ and that~$f$ is rational summable in~$K(x)$.

Suppose that~$f$ is rational summable in~$K(x)$, i.e., $f=\Delta_x(g)$ for some~$g\in K(x)$.
Over the field~$\ok$, we decompose~$g$ into the form~\eqref{EQ:dparfrac}. For all~$i, j$
with~$1\leq i \leq m$ and~$1\leq j \leq n_i$, the linearity of~$\Delta_x$ implies that
\[\Delta_x\left(\sum_{\ell=0}^{d_{i,j}} \frac{\alpha_{i,j,\ell}}{(x-(\beta_i+\ell))^{j}} \right)
=
\sum_{\ell=0}^{d_{i,j}+1} \frac{\tilde{\alpha}_{i,j,\ell}}{(x-(\tilde{\beta}_i +\ell))^{j}}, \]
where~$\tilde{\beta}_i = \beta_i -1$, $\tilde{\alpha}_{i,j, 0}={\alpha}_{i, j, 0}$,
$\tilde{\alpha}_{i, j, d_{i,j}+1} = -\alpha_{i, j, d_{i, j}}$,
and~$\tilde{\alpha}_{i, j, \ell} = \alpha_{i,j, \ell}-\alpha_{i, j, \ell-1}$
for~$1\leq \ell \leq d_{i, j}$. Then the
residue~$\dres_x(f, [\tilde{\beta}_i], j)=\sum_{\ell=0}^{d_{i, j}}\tilde{\alpha}_{i,j,\ell}=0$
for all~$i, j$. This completes the proof.
\end{proof}
\begin{remark}
Proposition~\ref{PROP:ratsum} is also known in literature  (see~\cite[Theorem 10]{Matusevich2000} or~\cite[Corollary 1]{Marshall2005}).
We have recast the known proofs in our terms to show the relevance of discrete residues.
\end{remark}

\subsection{$q$-discrete residues}\label{SUBSECT:qres}
Given a rational function, the $q$-analogue of Abramov's algorithm
in~\cite{Abramov1995b}
can decide whether it is rational $q$-summable or not.
Here, we present a $q$-analogue of Proposition~\ref{PROP:ratsum} in terms of a
$q$-discrete analogue of residues. To this end, we first recall
some terminology from~\cite{Abramov1975, Abramov1989, Abramov1995b}.

For an element~$\alpha \in \ok$, we call the subset~$\{\alpha \cdot q^i \mid i \in \bZ\}$ of~$\ok$
the~\emph{$q^\bZ$-orbit} of~$\alpha$
in~$\ok$, denoted by~$[\alpha]_q$. For a polynomial~$b\in K[x]\setminus K$, the value
\[\max\{i\in \bZ \mid \text{$\exists$  nonzero~$\alpha, \beta \in \ok$ such
that~$\alpha=q^i\cdot \beta$ and~$b(\alpha)=b(\beta)=0$}\}\]
is called the~\emph{$q$-dispersion} of~$b$ with respect to~$x$, denoted by~$\qdisp_x(b)$.
For~$b=\lambda x^n$ with~$\lambda \in K$ and~$n\in \bN\setminus\{0\}$,
we define~$\qdisp_x(b)=+\infty$.
The polynomial~$b$ is said to be~\emph{$q$-shift-free} with respect
to~$x$ if~$\qdisp_x(b)=0$.
Let~$f=a/b\in K(x)$ be such that~$a, b\in K[x]$ and~$\gcd(a, b)=1$.
Over the field~$\ok$, $f$ can be uniquely decomposed into the form
\begin{equation}\label{EQ:qparfrac}
f = c + xp_1 + \frac{p_2}{x^s} +
\sum_{i=1}^m \sum_{j=1}^{n_i} \sum_{\ell=0}^{d_{i,j}} \frac{\alpha_{i,j,\ell}}{(x-q^\ell \cdot \beta_i)^{j}},
\end{equation}
where~$c\in K$, $p_1, p_2 \in K[x]$,  $m, n_i\in \bN$ are nonzero,
$s, d_{i,j}\in \bN$, $\alpha_{i, j, \ell}, \beta_i\in \ok$, and~$\beta_i$'s are nonzero and
in distinct $q^\bZ$-orbits.

\begin{define}[$q$-discrete residue]\label{DEF:qres}
Let~$f\in K(x)$ be of the form~\eqref{EQ:qparfrac}.
The sum $\sum_{\ell=0}^{d_{i,j}}q^{-\ell \cdot j} \alpha_{i,j, \ell}$
is called the \emph{$q$-discrete residue} of~$f$ at
the~$q^\bZ$-orbit $[\beta_i]_q$ of multiplicity~$j$ (with respect to~$x$),
denoted by~$\qres_x(f, [\beta_i]_q, j)$. In addition, we call the constant~$c$ the
\emph{$q$-discrete residue} of~$f$ at infinity, denoted by~~$\qres_x(f, \infty)$.
\end{define}

We summarize some basic facts concerning rational $q$-summability in the next lemma.
For a detailed proof, one can see~\cite[\S 3]{Abramov1995b}.
\begin{lemma}\label{LM:ratqsum}
Let~$p, p_1, p_2\in K[x]$, $c\in K$, and~$s\in \bN\setminus \{0\}$ be as in \eqref{EQ:qparfrac}. Then
\begin{enumerate}
  \item $\deg_x(\Delta_{q, x}(p)) = \deg_x(p)$.
  \item If~$c$ is nonzero, then~$c$ is not rational $q$-summable in~$K(x)$.
  \item $f = xp_1 + p_2/x^s$ is rational $q$-summable in~$K(x)$.
\end{enumerate}
\end{lemma}

The following lemma is a $q$-analogue of Lemma~\ref{LM:dres} and its proof
proceeds in a similar way.
\begin{lemma}\label{LM:qres}
Let~$f=\sum_{\ell=0}^d \alpha_{\ell}/(x-q^\ell\cdot \beta )^s$ be such
that~$d, s\in \bN$,~$\alpha_{\ell}, \beta\in \ok$,
and~$\beta$ is nonzero.
Then~$f$ is rational $q$-summable in~$\ok(x)$ if and only if the
sum~$\sum_{\ell=0}^d q^{-\ell\cdot s}\alpha_{\ell}$ is zero, that is,
if and only if~$\qres_x(f, [\beta]_q, s)=0$.
\end{lemma}
\begin{proof}
Suppose that the sum~$\sum_{\ell=0}^d q^{-\ell\cdot s}\alpha_{\ell}$ is zero.
We show that~$f$ is rational $q$-summable in~$\ok(x)$. To this end,
we proceed by induction on~$d$.
In the base case when~$d=0$, $f$ is clearly rational $q$-summable since~$f=0$.
Suppose that the assertion holds for~$d=m$ with~$m\geq 0$. Note that
\begin{align*}
 \frac{\alpha_{m+1}}{(x-q^{m+1}\beta)^s} =
  \Delta_{q, x} \left(-\frac{\alpha_{m+1}}{(x-q^{m+1}\beta)^s}\right) +
    \frac{q^{-s}\alpha_{m+1}}{(x-q^m\beta)^s}.
\end{align*}
This implies that
\[\sum_{\ell=0}^{m+1} \frac{\alpha_{\ell}}{(x-q^{\ell}\beta)^s} =
\Delta_{q, x} \left(-\frac{\alpha_{m+1}}{(x-q^{m+1}\beta)^s}\right) +
\sum_{\ell=0}^{m} \frac{\tilde{\alpha}_{\ell}}{(x-q^{\ell}\beta)^s},\]
where~$\tilde{\alpha}_{\ell} = \alpha_{\ell}$ if~$0\leq \ell \leq m-1$
and~$\tilde{\alpha}_{m} = q^{-s}\alpha_{m+1} + \alpha_{m}$.
From the definition and assumption on the~$\alpha_\ell$'s,
the sum~$\sum_{\ell=0}^m q^{-\ell\cdot s}\tilde{\alpha}_{\ell}$ is zero.
The induction hypothesis then implies that there exists~$\tilde{g}\in \ok(x)$ such that
\[\sum_{\ell=0}^{m} \frac{\tilde{\alpha}_{\ell}}{(x-q^{\ell}\beta)^s} = \Delta_{q, x}(\tilde{g}).\]
So~$f=\Delta_{q, x}(g)$ with~$g=\tilde{g} - \alpha_{m+1}/(x-q^{m+1}\beta)^s\in \ok(x)$.
For the opposite implication, we assume to the contrary that the
sum~$\sum_{\ell=0}^d q^{-\ell \cdot s}\alpha_\ell$
is nonzero. Without loss of generality, we can assume that~$\alpha_0 \neq 0$.
Write~$\alpha_0 = \bar{\alpha}_0 + \tilde{\alpha}_0$ such that~$\tilde{\alpha}_0 +
\sum_{\ell=1}^{d} q^{-\ell\cdot s}\alpha_\ell =0$.
Since~$\sum_{\ell=0}^d q^{-\ell\cdot s}\alpha_\ell\neq 0$, $\bar{\alpha}_0 \neq 0$.
By the assertion shown above, there exists~$\tilde{g}\in \ok(x)$ such that
\[f = \frac{\bar{\alpha}_0}{(x-\beta)^s} + \Delta_{q, x}(\tilde{g}).\]
Since~$\qdisp_x((x-\beta)^s)=0$ and~$\bar{\alpha}_0\neq 0$,
${\bar{\alpha}_0}/{(x-\beta)^s}$ is not rational
summable by~\cite[Lemma 6.3]{Hardouin2008}.
Then~$f$ is not rational $q$-summable in~$\ok(x)$.
This completes the proof.
\end{proof}

\begin{prop}\label{PROP:ratqsum}
Let~$f=a/b\in K(x)$ be such that~$a, b\in K[x]$ and~$\gcd(a, b)=1$.
Then~$f$ is rational $q$-summable in~$K(x)$ if and only if the~$q$-discrete
residues $\qres_x(f, \infty)$
and~$\qres_x(f, [\beta]_q, j)$ are all zero for any $q^\bZ$-orbit~$[\beta]_q$ with~$\beta \neq 0$ and~$b(\beta)=0$
of any multiplicity~$j\in \bN$.
\end{prop}
\begin{proof}
Let~$f\in K(x)$ be decomposed into the form~\eqref{EQ:qparfrac}. If the residue of~$f$
at any $q^\bZ$-orbit~$[\beta]_q, \beta \neq 0$, of any multiplicity is zero, then Lemma~\ref{LM:qres} implies that
for all~$i, j$ with~$1\leq i \leq m$
and~$1\leq j \leq n_i$, the sum
\[\sum_{\ell=0}^{d_{i,j}} \frac{\alpha_{i,j,\ell}}{(x-q^{\ell}\beta_i)^{j}} = \Delta_{q, x}(g_{i,j})\quad
\text{for some~$g_{i, j}\in \ok(x)$.}\]
Since the rational function~$xp_1 + \frac{p_2}{x^s}$ in~\eqref{EQ:qparfrac} is rational $q$-summable
by Lemma~\ref{LM:ratqsum}, there exists~$u\in K(x)$ such that~$xp_1 + p_2/x^s = \Delta_{q, x}(u)$.
So~$f = \Delta_{q, x}(u + g)$, where~$g=\sum_{i=1}^m\sum_{j=1}^{n_i} g_{i,j}$.
As in Proposition~\ref{PROP:ratsum}, we see that~$g \in K(x)$ and therefore that
$f$ is rational~$q$-summable in~$K(x)$.

Suppose that~$f$ is rational $q$-summable in~$K(x)$, i.e., $f=\Delta_{q, x}(g)$ for some~$g\in K(x)$.
Over the field~$\ok$, we decompose~$g$ into the form~\eqref{EQ:qparfrac}. For all~$i, j$
with~$1\leq i \leq m$ and~$1\leq j \leq n_i$, the linearity of~$\Delta_{q, x}$ implies that
\[\Delta_{q, x}\left(\sum_{\ell=0}^{d_{i,j}} \frac{\alpha_{i,j,\ell}}{(x-q^{\ell}\beta_i)^{j}} \right)
=
\sum_{\ell=0}^{d_{i,j}+1} \frac{\tilde{\alpha}_{i,j,\ell}}{(x-q^{\ell}\tilde{\beta}_i)^{j}}, \]
where~$\tilde{\beta}_i = q^{-1}\beta_i$, $\tilde{\alpha}_{i,j, 0}=q^{-j}{\alpha}_{i, j, 0}$,
$\tilde{\alpha}_{i, j, d_{i,j}+1} = -\alpha_{i, j, d_{i, j}}$,
and~$\tilde{\alpha}_{i, j, \ell} = q^{-j}\alpha_{i,j, \ell}-\alpha_{i, j, \ell-1}$
for~$1\leq \ell \leq d_{i, j}$. Then the
residue~$\qres_x(f, [\tilde{\beta}_i]_q, j)=\sum_{\ell=0}^{d_{i, j}}q^{-\ell\cdot j}
\tilde{\alpha}_{i,j,\ell}=0$
for all~$i, j$. Since~$\Delta_{q, x}(c)=0$ for any constant~$c\in k$,
the residue of~$f$ at infinity is zero.
This completes the proof.
\end{proof}

\subsection{Residual forms}\label{SUBSECT:resform}
In terms of residues, we will present a normal form of a rational function in the
quotient space~$K(x)/\partial_x (K(x))$ with~$\partial_x\in \{D_x, \Delta_x, \Delta_{q, x}\}$.
Let~$f\in K(x)$. If~$f$ is of the form~\eqref{EQ:cparfrac}, then we can reduce it to
\[f = D_x(g) + r, \quad \text{where~$r=\sum_{i=1}^m \frac{\cres_x(f, \beta_i)}{x-\beta_i}$}.\]
Note that~$r$ actually lies in~$K(x)$. We call such an~$r$ the \emph{residual form}
of~$f$ with respect to~$D_x$. Similarly, residual forms with respect to~$\Delta_x$ and~$\Delta_{q, x}$
are respectively
\[r=\sum_{i=1}^m\sum_{j=1}^{n_i} \frac{\dres_x(f, [\beta_i], j)}{(x-\beta_i)^j}, \quad \text{where
~$\beta_i$'s in distinct $\bZ$-orbits}.\]
and
\[r=c + \sum_{i=1}^m\sum_{j=1}^{n_i} \frac{\qres_x(f, [\beta_i]_q, j)}{(x-\beta_i)^j},
\quad \text{where~$c\in K$ and~$\beta_i$'s in distinct $q^\bZ$-orbits}.\]
Such a residual form for a rational function is unique up to taking a different representative
from orbits. One can compute residual forms without introducing algebraic
extensions of~$K$ by algorithms
in~\cite{Hermite1872, Ostrogradsky1845, Horowitz1971, Paule1995b, Pirastu1995a, Pirastu1995b, Abramov1995b}.

\section{Algebraic functions}\label{SECT:algfuns}
As early as 1827, Abel  already observed that an algebraic function satisfies
a linear differential equation with polynomial coefficients~\cite[p.\ 287]{Abel1881}.
The annihilating differential equations are important in the study of algebraic
functions and their series expansions~\cite{Comtet1964, Chudnovsky1986, Chudnovsky1987}.
Algorithms for constructing differential annihilators for algebraic functions
have been developed in~\cite{Cockle1861, Harley1862, CormierSingerTragerUlmer2002, Nahay2003, BCLSS2007}.
It is not true that any algebraic function satisfies a linear  or a~$q$-linear recurrence.
In this section we characterize those algebraic functions that satisfy such equations and prove
a few lemmas concerning algebraic solutions of first order linear and~$q$-linear recurrences.
In the next section, we will see  how this restriction on algebraic solutions of such recurrences
is responsible for the essential difference between the continuous problems and the~($q$-)discrete ones.

Let~$k$ be an algebraically closed field of characteristic zero.
 Let~$q\in k$ be such that~$q^i\neq 1$ for any~$i\in \bZ\setminus\{0\}$.
Let~$k(t)$ be the field
of all rational functions in~$t$ over~$k$.
On the field~$k(t)$, we let~$D_t$, $S_t$, and~$Q_t$ denote the derivation, shift operator, and
$q$-shift operator with respect to~$t$, respectively. Let~$k(t)\langle D_t \rangle$
(resp.\ $k(t)\langle S_t \rangle$, $k(t)\langle Q_t \rangle$) denote
the ring of linear differential (resp.\ recurrence, $q$-recurrence) operators over~$k(t)$.
We recall the following fact for reference later. One can find its proof
in~\cite[p.\ 339]{Harley1862} or~\cite[p.\ 267]{Comtet1964}.

\begin{prop}\label{PROP:aflde}
Let~$\alpha(t)$ be an element of the algebraic closure of~$k(t)$. Then there exists
a nonzero operator~$L(t, D_t)\in k(t)\langle D_t\rangle$ such that~$L(\alpha)=0$.
\end{prop}

{As mentioned above, the situation is different if we consider the
linear \linebreak ($q$-)recurrence equations for
algebraic functions and the following results show that requiring an algebraic
function~$f$ to satisfy such a recurrence equation severely restricts~$f$.

\begin{prop}\label{PROP:afrde}
Let~$\alpha(t)$ be an element in the algebraic closure of~$k(t)$. If there exists
a nonzero operator~$L(t, S_t)\in k(t)\langle S_t \rangle$ such that~$L(\alpha)=0$,
then~$\alpha\in k(t)$.
\end{prop}

\begin{prop}\label{PROP:aflqe}
Let~$\alpha(t)$ be an element in the algebraic closure of~$k(t)$. If there exists
a nonzero operator~$L(t, Q_t)\in k(t)\langle Q_t \rangle$ such that~$L(\alpha)=0$,
then~$\alpha\in k(t^{1/n})$ for some positive integer~$n$.
\end{prop}

We have included complete proofs (and references to other proofs) of these results in the Appendix.

In the next section, algebraic functions will appear as residues
of bivariate rational functions and these functions will satisfy certain first order
linear ($q$-)recurrence relations. The following lemmas characterize the form of these functions.
Although these characterizations can be derived from Propositions~\ref{PROP:afrde} and~\ref{PROP:aflqe},
we will give more elementary proofs.
Abusing notation, we let~$S_t$ and~$Q_t$ denote arbitrary extensions of~$S_t$ and~$Q_t$ to automorphisms
of~$\overline{k(t)}$, the algebraic closure of~$k(t)$.

\begin{lemma}\label{LM:const} Let~$n$ be a positive integer.
\begin{enumerate}
\item[(i)] If~$f\in \overline{k(t)}$ and~$S_t^n(f) = f$, then~$f \in k$.
\item[(ii)] If~$f\in \overline{k(t)}$ and~$Q_t^n(f) = f$, then~$f \in k$.
\item[(iii)]  If~$f\in \overline{k(t)}$ and~$D_t(f)=0$, then~$f \in k$.
\end{enumerate}
\end{lemma}
\begin{proof}~$(i)$. We begin by showing that if~$f\in {k(t)}$ and~$S_t^n(f) = f$ then~$f \in k$.
If~$f\notin k$, then there exists an element~$a\in k$ such that~$a$ is a pole or zero of~$f$.
In this case the infinite set~$\{a + in \ |  \ i \in \bZ \}$ will also consist of poles or zeroes,
an impossibility since~$f$ is a rational function. Now assume that~$f\in \overline{k(t)}$ and~$S_t^n(f) = f$.
Let~$Y^\lambda + a_{\lambda-1} Y^{\lambda-1} + \ldots + a_0$ be the minimal polynomial of~$f$ over~$k(t)$.
We then have that~$Y^\lambda + S_t^n(a_{\lambda-1}) Y^{\lambda-1} + \ldots + S_t^n(a_0)$ is also the minimal
polynomial of~$f(t) = S_t^n(f(t))$.  Therefore~$S_t^n(a_i) = a_i$ for all~$i = \lambda -1, \ldots  , 0$.
This implies that the~$a_i \in k$. Since $k$ is algebraically closed,~$f \in k$.

$(ii)$. We again begin by showing that if~$f\in {k(t)}$ and~$Q_t^n(f) = f$ then~$f \in k$.
Assume~$f \notin k$ and let~$a \in k$ be a nonzero pole or zero of~$f$. We then have that
the set~$\{aq^{in} \ | \ i \in \bZ \}$ consists of poles or zeroes. Since~$q$ is not a root of unit,
this set is infinite and we get a contradiction as before.  Therefore,~$f = ct^m$ for some~$m \in \bZ$.
Since~$f(q^nt) = f(t)$, we have~$q^{nm} = 1$, a contradiction. Therefore~$f\in k$. Now assume
that~$f\in \overline{k(t)}$ and~$Q_t^n(f) = f$. An argument similar to that in 1.~shows that 2.~holds.

$(iii)$. This assertion follows from Lemma~3.3.2~(i) of~\cite[Chapter 3]{BronsteinBook} and the assumption
that~$k$ is algebraically closed.
\end{proof}

\begin{lemma}\label{LM:commuting} Let~$E \subset F$ be fields of characteristic zero  with~$F$
algebraic over~$E$.  Let~$\sigma$ be an automorphism of~$F$ such that~$\sigma(E) \subset E$ and
let~$\delta$ be a derivation of~$F$ such that~$\delta(E) \subset E$. If~$\delta\sigma(f) = \sigma\delta(f)$
for all~$f \in E$, then~$\delta\sigma(f) = \sigma\delta(f)$ for all~$f \in F$.\end{lemma}
\begin{proof} One can verify that~$\sigma^{-1}\delta\sigma$ is a derivation on~$F$ such
that~$\sigma^{-1}\delta\sigma(E) \subset E$. Therefore~$\sigma^{-1}\delta\sigma - \delta$ is
a derivation on~$f$ that is zero on~$E$. From the uniqueness of extensions of derivations to
algebraic extensions, we have that~$\sigma^{-1}\delta\sigma - \delta$ is zero on~$F$, which
yields the result.\end{proof}
}

\begin{lemma}\label{LM:cstdq}
Let~$\alpha(t)$ be an element in the algebraic closure of~$k(t)$.
If there exists a nonzero~$n\in \bN$
such that~$S_t^n(\alpha)= q^m\alpha$ for some~$m\in \bZ$,
then~$m=0$ and~$\alpha(t)\in k$.
\end{lemma}
\begin{proof} Let~$\delta = D_t$. Lemma~\ref{LM:commuting} implies
that~$S_t^n\delta = \delta S_t^n$ on~$\overline{k(t)}$. Therefore,~$S_t^n(\delta \alpha) = q^m \delta\alpha$.
One see that this implies that~$S_t^n(\delta\alpha/\alpha) = \delta\alpha/\alpha$, so by
Lemma~\ref{LM:const}~$\delta\alpha = c \alpha$ for some~$c \in k$. Assume that~$\alpha \notin k$ and
therefore that~$\delta\alpha \neq 0$ and~$c \neq 0$.
Let~$P(Y) = Y^\lambda + a_{\lambda-1} Y^{\lambda-1} + \ldots + a_0$ be the minimal polynomial
of~$\alpha$ over~$k(t)$. Applying~$\delta$ to~$P(\alpha)$, one sees that
\[Y^\lambda + \frac{\delta a_{\lambda-1} + (\lambda-1)c }{\lambda c} Y^{\lambda -1} + \ldots + \frac{\delta a_0}{\lambda c}\]
is also the minimal polynomial of~$\alpha$ over~$k(t)$.  Therefore
\[ \frac{\delta a_0}{a_0} = \lambda c .\]
Since~$a_0 \in k(t)$, we may write~$a_0 = d\prod(t-e_i)^{\mu_i}$, where~$d, e_i \in k, \mu_i \in \bZ$.
Therefore
\[\sum \frac{\mu_i}{t-e_i} = \lambda c\]
contradicting the uniqueness of partial fraction decomposition. This contradiction implies
that $\alpha \in k$.  From the equation~$S_t^n(\alpha)=q^m \alpha$ we get~$q^m=1$. Therefore~$m=0$ since~$q$
is not root of unity.
\end{proof}

\begin{lemma}\label{LM:intlin}
Let~$\alpha(t)$ be an element in the algebraic closure of~$k(t)$.
If there exists a nonzero~$n\in \bZ$
such that~$S_t^n(\alpha)-\alpha = m$ for some~$m\in \bZ$,
then~$\alpha(t)=\frac{m}{n} t + c$ for some~$c\in k$.
\end{lemma}
\begin{proof}
Let~$\beta(t) = \frac{m}{n} t$. Since~$S_t^n(\beta) - \beta = m$, we
have that $S_t^n(\alpha - \beta) - (\alpha - \beta) = 0$.  Therefore Lemma~\ref{LM:const}
implies that~$\alpha = \beta + c = \frac{m}{n}t + c$ for some~$c \in k$.
\end{proof}

\begin{lemma}\label{LM:cstqd}
Let~$\alpha(t)$ be an element in the algebraic closure of~$k(t)$.
If there exists a nonzero~$n\in \bZ$
such that~$Q_t^n(\alpha)-\alpha=m$ for some~$m\in \bZ$,
then~$m=0$ and~$\alpha(t)\in k$.
\end{lemma}
\begin{proof}Let$\delta = tD_t$.  One has
that~$\delta Q_t = Q_t \delta$ on~$k(t)$ so Lemma~\ref{LM:commuting}
implies that~$\delta Q_t = Q_t \delta$ on~$\overline{k(t)}$.  We then
also have~$\delta Q_t^n = Q_t^n \delta$ on~$\overline{k(t)}$ so~$Q_t^n(\delta \alpha) - \delta \alpha = 0$.
Lemma~\ref{LM:const} implies~$\delta\alpha \in k$. Suppose that~$\delta \alpha = c$ for~$c\in k$.
Then~$D_t(\alpha) = c/t$. If~$\text{Tr}: k(t)(\alpha) \rightarrow k(t)$ is the trace mapping, then
~$D_t(\text{Tr}(\alpha)) = \lambda c/t$ for some nonzero~$\lambda \in \bN$.
By Proposition~\ref{PROP:ratint}, we have~$\lambda c =0$ and then~$c=0$. Now~$\alpha\in k$
follows from the third assertion of Lemma~\ref{LM:const}.
\end{proof}

\begin{lemma}\label{LM:qintlin}
Let~$\alpha(t)$ be an element in the algebraic closure of~$k(t)$.
If there exists a nonzero~$n\in \bZ$
such that~$Q_t^n(\alpha)= q^m\alpha$ for some~$m\in \bZ$,
then~$\alpha(t)=c t^{\frac{m}{n}}$ for some~$c\in k$.
\end{lemma}
\begin{proof} Let~$\beta(t) = t^\frac{m}{n}$.  We then have that
\[Q_t^n(\frac{\alpha}{\beta}) = \frac{\alpha}{\beta}\]
so $\alpha/\beta = c \in k$ by Lemma~\ref{LM:const},
that is,~$\alpha = c t^\frac{m}{n}$.\end{proof}

\section{Telescopers}\label{SECT:telescoper}
In Section~\ref{SECT:residues}, we see that nonzero residues are the obstruction
for a rational function to being rational integrable (resp.\ summable, $q$-summable).
In this section, we consider whether we can use a linear operator,
a so-called~\emph{telescoper}, to remove this
obstruction if an extra parameter is available. The importance of telescopers in the study of special functions
and combinatorial identities have been
shown in the work by Zeilberger and his
collaborators~\cite{Zeilberger1990, Almkvist1990, WilfZeilberger1990a, WilfZeilberger1990b, Wilf1992}.

Let~$k(t, x)$ be the field of rational functions in~$t$ and~$x$ over~$k$.
On the field~$k(t, x)$, we have derivations~$D_t, D_x$, shift operators~$S_t, S_x$, and
$q$-shift operators~$Q_t, Q_x$. The linear operators used below will be in the
ring~$k(t)\langle D_t \rangle$, $k(t)\langle S_t \rangle$, or~$k(t)\langle Q_t \rangle$.
For a rational function~$f\in k(t, x)$, we wish to solve the Existence Problem for
Telescopers stated in the Introduction, that is, we want to decide the existence of
linear operators~$L(t, \partial_t)$ with~$\partial_t \in \{D_t, S_t, Q_t\}$
such that
\begin{equation}\label{EQ:tele}
L(t, \partial_t)(f)=\partial_x(g)
\end{equation}
for some~$g\in k(t, x)$
and~$\partial_x \in \{D_x, \Delta_x, \Delta_{q, x}\}$.
According to the different choices of~$L$
and~$\partial_x$, we have nine types of telescopers in general, see Table~\ref{tab:ninetelepb}.
\begin{center}
\renewcommand{\arraystretch}{1.2}
\tabcolsep4.5pt
\begin{table}[ht]
\begin{tabular}{|c|c|c|c|}
  \hline
  $(L, \partial_x)   $ & $D_x$ & $\Delta_x$ & $\Delta_{q, x}$ \\ \hline
  $k(t)\langle D_t \rangle $ & $L(t, D_t)(f) = D_x(g)$   & \underline{$L(t, D_t)(f) = \Delta_x(g)$} & \underline{$L(t, D_t)(f) = \Delta_{q, x}(g)$} \\
  $k(t)\langle S_t \rangle$ & \underline{$L(t, S_t)(f) = D_x(g)$}   & $L(t, S_t)(f) = \Delta_x(g)$ & \underline{$L(t, S_t)(f) = \Delta_{q, x}(g)$}\\
  $k(t)\langle Q_t \rangle$ & \underline{$L(t, Q_t)(f) = D_x(g)$}   & \underline{$L(t, Q_t)(f) = \Delta_x(g)$} & $L(t, Q_t)(f) = \Delta_{q, x}(g)$\\[0.05in]
  \hline
\end{tabular}
\caption{Nine different types of telescoping equations}\label{tab:ninetelepb}
\end{table}
\end{center}
\vspace{-0.5cm}
The existence problem of telescopers is related to the termination of Zeilberger-style
algorithms and has been studied in~\cite{AbramovLe2002, Abramov2003, Chen2005, CCFL2010}
but, to our knowledge, our results concerning telescopers of the six types underlined in
the above table are new.
In this section, we will present a unified way to solve this problem for rational functions
by using the knowledge in the previous sections. Before the investigation of the existence of telescopers,
we first present some preparatory lemmas for later use.
\begin{define}
Let~$\sim$ be an equivalence relation on a set~$R$ and~$\sigma: R\rightarrow R$ be a bijection.
The relation~$\sim$ is said to be~\emph{$\sigma$-compatible} if
\[\sigma(r_1) \sim \sigma(r_2) \, \, \Leftrightarrow \, \, r_1 \sim r_2 \quad \text{for all~$r_1, r_2\in R$.}\]
\end{define}

If the equivalence relation~$\sim$ is compatible with a bijection~$\sigma$ on~$R$,
then a bijection on the quotient set~$R/\sim$ can be naturally induced
by~$\sigma$, for which we still use the name~$\sigma$.
We denote by~$[t]$ the equivalence class of~$t$ in~$R/\sim$.

\begin{prop}\label{PROP:periodic}
Let~$\sigma: R\rightarrow R$ be a bijection and~$\sim$
be a~$\sigma$-compatible equivalence relation on the set~$R$.
Let~$T=\{[t_1], \ldots, [t_n]\}\subset R/\sim$. If for any~$i\in \{1, \ldots, n\}$,
there exists nonzero~$m_i\in \bN$ such that~$\sigma^{m_i}([t_i])\in T$,
then there exists nonzero~$m\in \bN$ such that~$\sigma^m([t_i])=[t_i]$ for all~$i\in \{1, \ldots, n\}$.
\end{prop}
\begin{proof}
Let~$\tilde{m}$ be the least common multiple of~$m_i$'s. Then~$\sigma^{\tilde{m}}$ is
a permutation on the finite set~$T$. Since any permutation on a finite set is idempotent, there exists an~$s\in \bN$
such that~$\sigma^{\tilde{m}s}$ is an identity on~$T$. Taking~$m=\tilde{m}s$ completes
the proof.
\end{proof}

We will specialize Proposition~\ref{PROP:periodic} to different bijections and equivalence relations.
The following examples show how to perform specializations.
\begin{example}\label{EX:intlin}
Let~$R$ be the algebraic closure of~$k(t)$. The equivalence relation~$\sim$ on~$R$ is defined by
$\alpha_1 \sim \alpha_2$ if and only if~$\alpha_1-\alpha_2\in \bZ$.
We take the shift mapping~$\sigma(\alpha(t))=\alpha(t+1)$ as the bijection.
Let~$T=\{[\alpha_1], \ldots, [\alpha_n]\}$ be such that for any~$i\in \{1, \ldots, n\}$,
$\sigma^{m_i}([\alpha_i])\in T$ for some nonzero~$m_i\in \bN$. By Proposition~\ref{PROP:periodic},
there exists nonzero~$m\in \bN$ such that~$\sigma^m(\alpha_i)-\alpha_i \in \bZ$ for all~$i\in \{1, \ldots, n\}$.
Applying Lemma~\ref{LM:intlin} to~$\alpha_i$ yields~$\alpha_i = \frac{n_i}{m} t +c_i$ for some~$n_i\in \bZ$
and~$c_i\in k$.
\end{example}

\begin{example}\label{EX:intqlin}
Let~$R$ be the algebraic closure of~$k(t)$. The equivalence relation~$\sim$ on~$R$ is defined by
$\alpha_1 \sim \alpha_2$ if and only
if~$\alpha_1/\alpha_2\in q^{\bZ}$. We take the $q$-shift mapping~$\sigma(\alpha(t))=\alpha(qt)$ as the bijection.
Let~$T=\{[\alpha_1]_q, \ldots, [\alpha_n]_q\}$ be such that for any~$i\in \{1, \ldots, n\}$,
$\sigma^{m_i}([\alpha_i])\in T$ for some nonzero~$m_i\in \bN$. By Proposition~\ref{PROP:periodic},
there exists nonzero~$m\in \bN$ such that~$\sigma^m(\alpha_i)/\alpha_i \in q^{\bZ}$
for all~$i\in \{1, \ldots, n\}$. Applying Lemma~\ref{LM:qintlin} to~$\alpha_i$
yields~$\alpha_i =  c_i t^{{n_i}/{m}}$ for some~$n_i\in \bZ$
and~$c_i\in k$.
\end{example}
\subsection{Existence of telescopers}\label{SUBSECT:existence}
The first result about the existence of telescopers was shown by Zeilberger in~\cite{Zeilberger1990}
based on the theory of holonomic D-modules. In the following, we will study the existence problems
from the residual point of view. For rational functions, the existence
of telescopers is related to the properties of residues and the commutativity between the residue
mappings and linear operators.

Starting from the simplest, we consider the telescoping relation~$L(t, D_t)(f)=D_x(g)$ for a given rational
function~$f\in k(t, x)$. Given~$\beta\in \overline{k(t)}$, view the residue mapping~$\cres_x(\underline{ \ \ }, \beta)$ as
a $\overline{k(t)}$-linear transformation
from~$\overline{k(t)}(x)$ to~$\overline{k(t)}$. For any~$\alpha, \beta\in \overline{k(t)}$, we have
\[D_t\left(\frac{\alpha}{x-\beta}\right) = \frac{D_t(\alpha)}{x-\beta} + \frac{\alpha
D_t(\beta)}{(x-\beta)^2}.\]
Then~$\cres_x(D_t(f), \beta) = D_t(\cres_x(f, \beta))$ for any~$f\in \overline{k(t)}(x)$ and~$\beta\in \overline{k(t)}$.
Assume that~$f=a/b$ with~$a, b \in k[t, x]$ and~$\gcd(a, b)=1$. Let~$\beta_1, \ldots, \beta_m$ be
the roots of~$b$ in~$\overline{k(t)}$. For each root~$\beta_i$, the continuous residue
$\cres_x(f, \beta_i) \in \overline{k(t)}$ is annihilated by a linear differential
operator~$L_{i}\in k(t)\langle D_t \rangle$ by Proposition~\ref{PROP:aflde}. Let~$L(t, D_t)$ be
the least common left multiple (LCLM) of the~$L_i$'s. Then we have~$L(\cres_x(f, \beta_i))= \cres_x(L(f), \beta_i)=0$
for all~$i$ with~$1\leq i \leq m$. So~$L(f)$ is rational integrable with respect to~$x$
by Proposition~\ref{PROP:ratint}. In summary, we have the following theorem.
\begin{theorem}\label{THM:cc}
For any~$f\in k(t, x)$, there exists a nonzero
operator~$L\in k(t)\langle D_t \rangle$ such that~$L(f)= D_x(g)$ for some~$g\in k(t, x)$.
\end{theorem}
However, the situation in other cases turns out to be more involved.
For  the rational function~$f=1/(t^2+x^2)$, Abramov and Le~\cite{Le2001, AbramovLe2002} showed that
there is no telescoper in~$k(t)\langle S_t \rangle$ such that~$L(f)=\Delta_x(g)$ for any~$g\in k(t, x)$.
In other cases, there are two main reasons for non-existence: one is the non-commutativity between
linear  operators~$\partial_t\in \{D_t, S_t, Q_t\}$ and residue mappings, the other is
that not all algebraic functions would satisfy linear ($q$)-recurrence relations.
So it is natural that rational functions are of special forms if telescopers
exist.

Let~$f\in k(t, x)$ and~$\partial_x\in \{D_x, \Delta_x, \Delta_{q, x}\}$.
Then~$f = \partial_x(g) + r$ with~$g, r\in k(t, x)$
and~$r$ being the residual form of~$f$ with respect to~$\partial_x$ {(see Section~\ref{SUBSECT:resform})}.
Since linear operators~$L(t, \partial_t)$ with~$\partial_t\in \{D_t, S_t, Q_t\}$
commute with  the linear operator~$\partial_x\in \{D_x, \Delta_x, \Delta_{q, x}\}$, a rational function has a telescoper if and only if its residual form does.
From now on, we always assume that the given rational function is in its residual form.
We will also use the fact~\cite[Lemma 1]{AbramovLe2002} that the sum~$f_1+f_2$ has a telescoper if both~$f_1$ and~$f_2$ do.
To be more precise, if~$L_1, L_2$ are telescopers for~$f_1, f_2$, respectively, then the LCLM of~$L_1, L_2$
is a telescoper for~$f_1+f_2$.

\subsubsection{Telescopers with respect to~$D_x$}\label{SUBSECT:existencediff}
Let~$f\in k(t, x)$ be a residual form, that is,
\begin{equation}\label{EQ:cresform}
    f = \sum_{i=1}^m \frac{\alpha_i}{x-\beta_i}, \quad \text{where~$\alpha_i, \beta_i\in \overline{k(t)}$ and the~$\beta_i$ are pairwise distinct.}
\end{equation}
\begin{theorem}\label{THM:dc}
Let~$f\in k(t, x)$ be as in~\eqref{EQ:cresform}. Then~$f$
has a telescoper~$L$ in~$k(t)\langle S_t\rangle$ such that~$L(t, S_t)(f)= D_x(g)$ for some~$g\in k(t, x)$
if and only if all the~$\beta_i$ are in~$k$.
\end{theorem}
\begin{proof}
Suppose that there exists a nonzero~$L\in k(t)\langle S_t\rangle$ such that~$L(t, S_t)(f)=D_x(g)$
for some~$g\in k(t, x)$. Write~$L=\sum_{\ell=0}^{\rho} e_{\ell}S_t^{\ell}$ with~$e_{\ell}\in k(t)$
and~{$e_\rho = 1$}.
Then
\[L(f) = \sum_{\ell=0}^{\rho} \sum_{i=1}^m \frac{e_{\ell}S_t^{\ell}(\alpha_i)}{x-S_t^{\ell}(\beta_i)}.\]
Assume that~$\ell_0$ is the first index in~$\{0, 1, \ldots, \rho\}$ such that~$e_{\ell_0}\neq 0$.
Since~$L(f)$ is rational integrable in~$k(t, x)$ with respect to~$D_x$, all residues of~$L(f)$ are
zero by Proposition~\ref{PROP:ratint}. In particular, the
set~$T=\{S_t^{\ell_0}(\beta_1), \ldots, S_t^{\ell_0}(\beta_m)\}$ satisfies the property
that for any~$i\in \{1, \ldots, m\}$,
there exists nonzero~$m_i\in \bN$ such that~$S_t^{\ell_0+m_i}(\beta_i)\in T$.
By taking equality as the equivalence
relation and the shift mapping as the bijection in Proposition~\ref{PROP:periodic},
there exists nonzero~$m\in \bN$ such
that~$S_t^{\ell_0+m}(\beta_i)= \beta_i$ for all~$i\in \{1, \ldots, m\}$.  By Lemma~\ref{LM:const}~(i) and the assumption
that~$k$ is algebraically closed, all the~$\beta_i$ are in~$k$.

For the opposite implication, it suffices to show that each fraction~$\alpha_i/(x-\beta_i)$ with~$\beta_i\in k$
has a telescoper in~$k(t)\langle S_t\rangle$. According to the process of partial fraction decomposition,
~$\alpha_i\in k(t)(\beta_i)$ for any~$i$ with~$1\leq i\leq m$. Then~$\alpha_i\in k(t)$, which is annihilated
by the operator~$L_i = S_t - {\alpha_i(t+1)}/{\alpha_i(t)}$. Moreover,
$L_i(\alpha_i/(x-\beta_i))=L_i(\alpha_i)/(x-\beta_i)=0$. So the LCLM of the~$L_i$'s is a
telescoper for~$f$. This completes the proof.
\end{proof}

\begin{theorem}\label{THM:qc}
Let~$f\in k(t, x)$ be as in~\eqref{EQ:cresform}. Then~$f$
has a telescoper~$L$ in~$k(t)\langle Q_t\rangle$ such that~$L(t, Q_t)(f)= D_x(g)$ for some~$g\in k(t, x)$
if and only if all the~$\beta_i$ are in~$k$.
\end{theorem}
\begin{proof} The proof proceeds in a similar way as above replacing~$S_t$ by~$Q_t$
and  Lemma~\ref{LM:const}~(i) by~Lemma~\ref{LM:const}~(ii).
\end{proof}
\begin{example}
Let~$f=1/(x+t)$. Since the root of~$x+t$ in~$\overline{k(t)}$ is~$t$, which is not in~$k$, $f$ has no
telescoper in either~$k(t)\langle S_t\rangle$ or~$k(t)\langle Q_t\rangle$
with respect to~$D_x$ by Theorems~\ref{THM:dc} and~\ref{THM:qc}.
\end{example}

\subsubsection{Telescopers with respect to~$\Delta_x$}\label{SUBSECT:existenceshift}
Let~$f\in k(t, x)$ be of the form
\begin{equation}\label{EQ:dresform}
    f = \sum_{i=1}^m\sum_{j=1}^{n_i} \frac{\alpha_{i, j}}{(x-\beta_i)^j},
\end{equation}
where~$\alpha_{i,j}, \beta_i\in \overline{k(t)}$, $\alpha_{i, n_i}\neq 0$, and the~$\beta_i$ are
in distinct~$\bZ$-orbits.
\begin{theorem}\label{THM:cd}
Let~$f\in k(t, x)$ be as in~\eqref{EQ:dresform}. Then~$f$
has a telescoper~$L$ in~$k(t)\langle D_t\rangle$ such that~$L(t, D_t)(f)= \Delta_x(g)$
for some~$g\in k(t, x)$ if and only if all the~$\beta_i$ are in~$k$.
\end{theorem}
\begin{proof}
Suppose that there exists a nonzero~$L\in k(t)\langle D_t\rangle$ such that~$L(t, D_t)(f)=\Delta_x(g)$
for some~$g\in k(t, x)$. Write~$L=\sum_{\ell=0}^{\rho} e_{\ell}D_t^{\ell}$ with~$e_{\ell}\in k(t)$.
By induction on~$\ell$, we get
\[D_t^{\ell}\left(\frac{\alpha_{i, n_i}}{(x-\beta_i)^{n_i}}\right) =
\frac{(n_i)_{\ell}\alpha_{i, n_i}(D_t(\beta_i))^{\ell}}{(x-\beta_i)^{n_i+\ell}} + \, \, \text{lower terms,}\]
where~$(n_i)_{\ell} = n_i (n_i+1) \cdots (n_i+\ell-1)$. Then we have
\[L(f) = \sum_{i=1}^m \frac{(n_i)_{\rho}\alpha_{i, n_i}(D_t(\beta_i))^{\rho}}{(x-\beta_i)^{n_i+\rho}}
+\,\, \text{lower terms.}\]
Since~$L(f)$ is rational summable with respect to~$\Delta_x$ and the~$\beta_i$ are in distinct~$\bZ$-orbits,
we get~$(n_i)_{\rho}\alpha_{i, n_i}(D_t(\beta_i))^{\rho}=0$ for all~$i\in \{1, \ldots, m\}$ by Proposition~\ref{PROP:ratsum}.
Since~$\alpha_{i, n_i}\neq 0$ and~$(n_i)_{\rho}>0$, $D_t(\beta_i)=0$, which implies that~$\beta_i\in k$
{by Lemma~\ref{LM:const}~(iii)}.

For the opposite implication, the proof is similar to that of Theorem~\ref{THM:dc}.
Let~$L_{i,j}$ be the operator~$D_t - D_t(\alpha_{i, j})/\alpha_{i,j}\in k(t)\langle D_t \rangle$. Then the LCLM of the~$L_{i, j}$ is
a telescoper for~$f$ with respect to~$\Delta_x$.
\end{proof}

\begin{example}\label{EX:transcendental} Let
\[f = \frac{1}{x^2 - t} =\frac{1}{2\sqrt{t}}\left(\frac{1}{x - \sqrt{t}} - \frac{1}{x +\sqrt{t}}\right).\]
Note that $f$ is already in residual form with respect to~$\Delta_x$. By
Theorem~\ref{THM:cd}, there is no linear differential
operator~$L(t,D_t) \in k(t)\langle D_t\rangle$ and $g \in k(t,x)$ such that $L(t,D_t)f = \Delta_x(g)$. Furthermore,
Proposition~3.1 in~\cite{Hardouin2008} and the descent argument similar to that
given in the proof of Corollary~3.2 of~\cite{Hardouin2008} (or Section 1.2.1
of~\cite{DH2011}) implies that the sum
\[F(t,x) = \sum_{i=1}^{x -1}\frac{1}{{ i}^2-t} \ \  \mbox { (satisfying~$S_x(F) - F = f$) }\]
 satisfies no polynomial differential equation $P(t,x,F,D_tF,D_t^2F, \ldots ) = 0$.
\end{example}

The following theorem is the same as~\cite[Theorem 1]{AbramovLe2002}. We give an alternative
proof using the knowledge developed in previous sections.

\begin{theorem}\label{THM:dd}
Let~$f\in k(t, x)$ be as in~\eqref{EQ:dresform}. Then~$f$
has a telescoper~$L$ in~$k(t)\langle S_t\rangle$ such that~$L(t, S_t)(f)= \Delta_x(g)$
for some~$g\in k(t, x)$ if and only if all the~$\beta_i=r_i t +c_i$ with~$r_i\in \bQ$ and~$c_i\in k$.
\end{theorem}
\begin{proof}
Suppose that there exists a nonzero~$L\in k(t)\langle S_t\rangle$ such that~$L(t, S_t)(f)=\Delta_x(g)$
for some~$g\in k(t, x)$. Write~$L=\sum_{\ell=0}^{\rho} e_{\ell}S_t^{\ell}$ with~$e_{\ell}\in k(t)$
and~$e_0\neq 0$.
For any~$\lambda \in \{1, \ldots, m\}$, we consider the rational function
\[f_{\lambda} = \sum_{i=1}^m \frac{\alpha_{i, n_{\lambda}}}{(x-\beta_i)^{n_{\lambda}}},\quad \text{where~$\alpha_{\lambda, n_{\lambda}}\neq 0$ by assumption} .\]
Without loss of generality, we may assume that the other~$\alpha_{i, n_{\lambda}}$ with~$i\neq \lambda$ are
also nonzero. Since the shift operators~$S_t, S_x$ preserve the multiplicity, we
have~$L(f_{\lambda})=\Delta_x(g_{\lambda})$ for some~$g_{\lambda}\in k(t, x)$. By Proposition~\ref{PROP:ratsum},
all the residues of~$L(f_{\lambda})$ are zero. We now use the notation and analysis of Example~\ref{EX:intlin}.
We see that the
set~$T=\{[\beta_1], \ldots, [\beta_m]\}$ satisfies the property that for any~$i \in \{1, \ldots, m\}$, there exists
a nonzero~$m_i$ such that~$S_t^{m_i}([\beta_i])\in T$. As in  Example~\ref{EX:intlin},
we conclude that~$\beta_i=\frac{p_i}{m} t +c_i$ with~$p_i, m\in \bZ$ and~$c_i\in k$.

The opposite implication follows from the fact that the linear
operator
\[L_{i, j}=\alpha_{i, j}(t)S_t^{m}-\alpha_{i,j}(t+m)\]
is a telescoper for the fraction~$f_{i, j} = \alpha_{i, j}/(x-(\frac{p_i}{m} t +c_i))^j$
with respect to~$\Delta_x$ since~$\dres(L_{i, j}(f_{i, j}), [\frac{p_i}{m} t +c_i], j)=0$.
Then the LCLM of the~$L_{i, j}$ is a telescoper for~$f$ with respect to~$\Delta_x$.
\end{proof}

\begin{theorem}\label{THM:qd}
Let~$f\in k(t, x)$ be as in~\eqref{EQ:dresform}. Then~$f$
has a telescoper~$L$ in~$k(t)\langle Q_t\rangle$ such that~$L(t, Q_t)(f)= \Delta_x(g)$
for some~$g\in k(t, x)$ if and only if all the~$\beta_i$ are in~$k$.
\end{theorem}
\begin{proof}
Suppose that there exists a nonzero~$L\in k(t)\langle Q_t\rangle$ such that~$L(t, Q_t)(f)=\Delta_x(g)$
for some~$g\in k(t, x)$. Write~$L=\sum_{\ell=0}^{\rho} e_{\ell}Q_t^{\ell}$ with~$e_{\ell}\in k(t)$
and~$e_0\neq 0$.
For any~$\lambda \in \{1, \ldots, m\}$, we consider the rational function
\[f_{\lambda} = \sum_{i=1}^m \frac{\alpha_{i, n_{\lambda}}}{(x-\beta_i)^{n_{\lambda}}},\quad \text{where~$\alpha_{\lambda, n_{\lambda}}\neq 0$ by assumption} .\]
Without loss of generality, we may assume that the other~$\alpha_{i, n_{\lambda}}$ with~$i\neq \lambda$ are
also nonzero. Since the operators~$Q_t, S_x$ preserve the multiplicity, we
have~$L(f_{\lambda})=\Delta_x(g_{\lambda})$ for some~$g_{\lambda}\in k(t, x)$.
By Proposition~\ref{PROP:ratsum},
all the residues of~$L(f_{\lambda})$ are zero.
We shall again use the reasoning and notation in Example~\ref{EX:intlin} where~$[ \ \ ]$
is an equivalence class of the equivalence relation that~$\alpha_1\sim \alpha_2$
in~$\overline{k(t)}$ if~$\alpha_1-\alpha_2\in \bZ$.  In particular, the
set~$T=\{[\beta_1], \ldots, [\beta_m]\}$ satisfies the property that for any~$i \in \{1, \ldots, m\}$, there exists
a nonzero~$m_i$ such that~$Q_t^{m_i}([\beta_i])\in T$. Taking the shift mapping~$Q_t$ as  the bijection,
Proposition~\ref{PROP:periodic} and Lemma~\ref{LM:cstqd} imply
that~$\beta_i\in k$ for all~$i$ with~$1\leq i \leq m$.

The opposite implication follows from the fact that the linear
operator
\[L_{i, j}=\alpha_{i, j}(t)Q_t-\alpha_{i,j}(qt)\]
is a telescoper for the fraction~$f_{i, j} = \alpha_{i, j}/(x-\beta_i)^j$
with respect to~$\Delta_x$ since $\dres(L_{i, j}(f_{i, j}), [\beta_i], j)=0$.
Then the LCLM of the~$L_{i, j}$ is a telescoper for~$f$ with respect to~$\Delta_x$.
\end{proof}
\subsubsection{Telescopers with respect to~$\Delta_{q, x}$}\label{SUBSECT:existenceqshift}
Let~$f\in k(t, x)$ be of the form
\begin{equation}\label{EQ:qresform}
    f = c + \sum_{i=1}^m\sum_{j=1}^{n_i} \frac{\alpha_{i, j}}{(x-\beta_i)^j},
\end{equation}
where~$c\in k(t)$, $\alpha_{i,j}, \beta_i\in \overline{k(t)}$, $\alpha_{i, n_i}\neq 0$, and the~$\beta_i$ are in distinct~$q^\bZ$-orbits.
\begin{theorem}\label{THM:cq}
Let~$f\in k(t, x)$ be as in~\eqref{EQ:qresform}. Then~$f$
has a telescoper~$L$ in~$k(t)\langle D_t\rangle$ such that~$L(t, D_t)(f)= \Delta_{q, x}(g)$
for some~$g\in k(t, x)$ if and only if all the~$\beta_i$ are in~$k$.
\end{theorem}
\begin{proof}
The proof proceeds in the same way as that in Theorem~\ref{THM:cd}.
\end{proof}

\begin{theorem}\label{THM:dq}
Let~$f\in k(t, x)$ be as in~\eqref{EQ:qresform}. Then~$f$
has a telescoper~$L$ in~$k(t)\langle S_t\rangle$ such that~$L(t, S_t)(f)= \Delta_{q, x}(g)$
for some~$g\in k(t, x)$ if and only if all the~$\beta_i$ are in~$k$.
\end{theorem}
\begin{proof}
Suppose that there exists a nonzero~$L\in k(t)\langle S_t\rangle$ such that~$L(t, S_t)(f)=\Delta_{q, x}(g)$
for some~$g\in k(t, x)$. Write~$L=\sum_{\ell=0}^{\rho} e_{\ell}S_t^{\ell}$ with~$e_{\ell}\in k(t)$
and~$e_0\neq 0$.
For any~$\lambda \in \{1, \ldots, m\}$, we consider the rational function
\[f_{\lambda} = \sum_{i=1}^m \frac{\alpha_{i, n_{\lambda}}}{(x-\beta_i)^{n_{\lambda}}},\quad \text{where~$\alpha_{\lambda, n_{\lambda}}\neq 0$ by assumption} .\]
Without loss of generality, we may assume that the other~$\alpha_{i, n_{\lambda}}$ with~$i\neq \lambda$ are
also nonzero. Since the operators~$S_t, Q_x$ preserve the multiplicity, we
have~$L(f_{\lambda})=\Delta_{q, x}(g_{\lambda})$ for some~$g_{\lambda}\in k(t, x)$.
By Proposition~\ref{PROP:ratqsum},
all the residues of~$L(f_{\lambda})$ are zero.
We now use the reasoning and notation in Example~\ref{EX:intqlin}. In particular, the
set~$T=\{[\beta_1]_q, \ldots, [\beta_m]_q\}$ satisfies that for any~$i \in \{1, \ldots, m\}$, there exists
a nonzero~$m_i$ such that~$S_t^{m_i}([\beta_i]_q)\in T$. Taking the shift mapping~$S_t$ as bijection,
Proposition~\ref{PROP:periodic} and Lemma~\ref{LM:cstdq} imply
that~$\beta_i\in k$ for all~$i$ with~$1\leq i \leq m$.

The opposite implication follows from the fact that~$c(t)$ is
annihilated by the operator~$L_0 =c(t)S_t - c(t+1)$ and
the linear operator
\[L_{i, j}=\alpha_{i, j}(t)S_t-\alpha_{i,j}(t+1)\]
is a telescoper for the fraction~$f_{i, j} = \alpha_{i, j}/(x-\beta_i)^j$
with respect to~$\Delta_{q, x}$ since $\dres(L_{i, j}(f_{i, j}), [\beta_i]_q, j)=0$.
Then the LCLM of the~$L_0$ and~$L_{i, j}$ is a telescoper for~$f$ with respect to~$\Delta_{q, x}$.
\end{proof}

The following theorem is a $q$-analogue of Theorem~\ref{THM:dd}, which
has also been shown in~\cite[Theorem 1]{Le2001}.
\begin{theorem}\label{THM:qq}
Let~$f\in k(t, x)$ be as in~\eqref{EQ:qresform}. Then~$f$
has a telescoper~$L$ in~$k(t)\langle Q_t\rangle$ such that~$L(t, Q_t)(f)= \Delta_{q, x}(g)$
for some~$g\in k(t, x)$ if and only if all the~$\beta_i=c_i t^{r_i}$ with~$r_i\in \bQ$ and~$c_i\in k$.
\end{theorem}
\begin{proof}
Suppose that there exists a nonzero~$L\in k(t)\langle Q_t\rangle$ such that~$L(t, Q_t)(f)=\Delta_{q, x}(g)$
for some~$g\in k(t, x)$. Write~$L=\sum_{\ell=0}^{\rho} e_{\ell}Q_t^{\ell}$ with~$e_{\ell}\in k(t)$
and~$e_0\neq 0$.
For any~$\lambda \in \{1, \ldots, m\}$, we consider the rational function
\[f_{\lambda} = \sum_{i=1}^m \frac{\alpha_{i, n_{\lambda}}}{(x-\beta_i)^{n_{\lambda}}},\quad \text{where~$\alpha_{\lambda, n_{\lambda}}\neq 0$ by assumption} .\]
Without loss of generality, we may assume that the other~$\alpha_{i, n_{\lambda}}$ with~$i\neq \lambda$ are
also nonzero. Since the $q$-shift operators~$Q_t, Q_x$ preserve the multiplicity, we
have~$L(f_{\lambda})=\Delta_{q, x}(g_{\lambda})$ for some~$g_{\lambda}\in k(t, x)$.
By Proposition~\ref{PROP:ratqsum},
all the residues of~$L(f_{\lambda})$ are zero. In particular, the
set~$T=\{[\beta_1]_q, \ldots, [\beta_m]_q\}$ satisfies that for any~$i \in \{1, \ldots, m\}$, there exists
a nonzero~$m_i$ such that~$Q_t^{m_i}([\beta_i]_q)\in T$. By the analysis in Example~\ref{EX:intqlin},
we conclude that~$\beta_i=c_it^{{p_i}/{m}}$ with~$p_i, m\in \bZ$ and~$c_i\in k$.

The opposite implication follows from the fact that~$c(t)$ is annihilated by
the operator~$L_0 =cS_t - c(t+1)$ and
the linear operator
\[L_{i, j}=\alpha_{i, j}(t)Q_t^{m}-q^{-jp_i}\alpha_{i,j}(q^mt)\]
is a telescoper for the fraction~$f_{i, j} = \alpha_{i, j}/(x-(c_it^{{p_i}/{m}}))^j$
with respect to~$\Delta_{q, x}$ since~$\qres(L_{i, j}(f_{i, j}), [c_it^{{p_i}/{m}}]_q, j)=0$.
Then the LCLM of the~$L_0$ and~$L_{i, j}$ is a telescoper for~$f$ with respect to~$\Delta_{q, x}$.
\end{proof}

The necessary and sufficient conditions for the existence of telescopers enable
us to decide the termination of the Zeilberger algorithm for rational-function
inputs. After reducing the given rational function into a residual form, one
can detect the existence by investigating the denominator. For instance, we could
check whether the denominator factors into two univariate polynomials respectively in~$t$
and~$x$ in the case when~$\partial_t=D_t$ and~$\partial_x=\Delta_x$. Combining
the existence criteria with the Zeilberger algorithm yields a complete algorithm
for creative telescoping with rational-function inputs.

\subsection{Characterization of telescopers}\label{SUBSECT:chartele}
We have shown that telescopers exist for a special class
of rational functions.
Now, we will characterize the linear differential
and ($q$-)recurrence operators that could be telescopers
for rational functions. Using such a characterization, we will give a direct algebraic proof
of a theorem of Furstenberg  stating that the diagonal of a rational power series in two variables
is algebraic~\cite{Furstenberg1967}. In all of these considerations, residues are still the key.

For a rational function~$f\in k(t, x)$, all of the telescopers for~$f$
in~$k(t)\langle D_t\rangle$ form a left ideal in~$k(t)\langle D_t\rangle$, denoted by~$\mathcal{T}_f$.
Since the ring~$k(t)\langle D_t\rangle$ is a left Euclidean domain, the monic telescoper
of minimal order generates the left ideal~$\mathcal{T}_f$, and we
call this generator~\emph{the minimal telescoper} for~$f$.

\begin{theorem}\label{THM:telecc}
Let~$L(t, D_t)$ be a linear differential operator in~$k(t)\langle D_t\rangle$.
Then~$L$ is a telescoper for some~$f\in k(t, x)\setminus D_x(k(t, x))$
such that~$L(f)= D_x(g)$ with~$g\in k(t, x)$ if and only if $L(y(t))=0$ has
a nonzero solution algebraic over~$k(t)$. Moreover, if~$L$ is the
minimal telescoper for~$f$, then all solutions of~$L(y(t))=0$ are algebraic over~$k(t)$.
\end{theorem}
\begin{proof}
Suppose that there exists~$f\in k(t, x)\setminus D_x(k(t, x))$ such that~$L(f)= D_x(g)$ for some~$g\in k(t, x)$.
Since~$f$ is not rational integrable with respect to~$x$, $f$ has a nonzero residue by Proposition~\ref{PROP:ratint}.
Since~$L$ is a telescoper for~$f$ with respect to~$D_x$,
$L$ vanishes at all residues of~$f$. So~$L(y(t))=0$ has a nonzero
algebraic solution in~$\overline{k(t)}$ because any residue of a rational function
in~$k(t, x)$ is algebraic over~$k(t)$.

Conversely, if~$\alpha\in \overline{k(t)}$ is a nonzero algebraic solution of~$L(y(t))=0$ with minimal
polynomial~$P\in k[t, x]$, then~$L$ is a telescoper for the rational function~$f=xD_x(P)/P$ with respect to~$D_x$.

Let~$a/b\in k(t, x)$ be the residual form of~$f$ with respect to~$D_x$. All of the residues of~$a/b$
are roots of the polynomial~$R(t, z) = \mbox{resultant}_x(b, a-zD_x(b))\in k(t)[z]$. By
the method in~\cite[\S 2]{CormierSingerTragerUlmer2002}, one can construct the minimal operator~$L_R$
in~$k(t)\langle D_t\rangle$ such that~$L_R(\alpha(t))=0$ for all roots of~$R$ in~$\overline{k(t)}$.
Moreover, the solutions space of~$L_R$ is spanned by the roots of~$R$. Since~$L_R$ vanishes
at all residues of~$f$, $L_R$ is a telescoper for~$f$. If~$L$ is the minimal telescoper for~$f$,
then~$L$ divides~$L_R$ on the right. Thus, all solutions of~$L(y(t))=0$ are solutions of~$L_R(y(t))=0$,
and therefore algebraic over~$k(t)$.
\end{proof}
The diagonal~$\operatorname{diag}(f)$ of a formal power series~$f=\sum_{i, j\geq 0} f_{i, j}t^ix^j\in k[[t, x]]$ is
defined by
\[\operatorname{diag}(f) = \sum_{i\geq 0} f_{i, i} t^i\in k[[t]].\]
Using the characterization of telescopers in Theorem~\ref{THM:telecc}, we now give a proof of a theorem of Furstenberg
that the diagonal of a rational power series in two variables is algebraic~\cite{Furstenberg1967}.
For other proofs, see the papers~\cite{Fliess1974, Gessel1980, Haiman1993}
and Stanley's book~\cite[Theorem 6.3.3]{Stanley1999}.

Let~$\mathcal {F}=k((x))$ be the quotient field of~$k[[x]]$ and~$\mathcal {F}[[t]]$ be
the formal power series over~$\mathcal {F}$. We use the notation~$[x^{-1}](a)$ to denote the
coefficient of~$x^{-1}$ in~$a\in \mathcal {F}$. For a formal power series~$g=\sum_{i\geq 0} a_i(x)t^i \in \mathcal {F}[[t]]$,
we define
\[[x^{-1}](g)=\sum_{i\geq 0} ([x^{-1}](a_i))t^i\in k[[t]],\]
and two derivations
\[D_t(g) = \sum_{i\geq 0} i a_i(x)  t^{i-1}, \quad D_x(g) = \sum_{i\geq 0} D_x(a_i)t^{i}.\]
The ring~$\mathcal {F}[[t]]$ then becomes a $k[t, x]\langle D_t, D_x\rangle$-module.
By definition, we have
\[[x^{-1}](D_t(g)) = D_t([x^{-1}](g)) \quad  \text{and} \quad [x^{-1}](t^i(g)) = t^i([x^{-1}](g))
\]
for all~$i\in \bN$. By induction, we have~$L([x^{-1}](g))=[x^{-1}](L(g))$ for all~$L\in k[t]\langle D_t\rangle$.
Since~$[x^{-1}](D_x(a))=0$ for any~$a\in \mathcal {F}$, we get~$[x^{-1}](D_x(g))=0$ for any~$g\in \mathcal {F}[[t]]$.
Let~$f=\sum_{i, j\geq 0} f_{i, j}t^ix^j$ be a formal power series
in~$k[[t, x]]$. Then~$F=f(x, t/x)/x$ is in~$\mathcal {F}[[t]]$. Applying~$[x^{-1}]$ to~$F$ yields
\[[x^{-1}](F) = [x^{-1}](\sum_{i, j\geq 0} f_{i, j} x^{i-j-1}t^j) = \sum_{j\geq 0} f_{j, j} t^j
= \operatorname{diag}(f).\]
If~$L\in k[t]\langle D_t\rangle$ be
such that~$L(F)=D_x(G)$ for some~$G\in \mathcal {F}[[t]]$, then applying~$[x^{-1}]$ to both sides
of~$L(F)=D_x(G)$ yields~$L(\operatorname{diag}(f))=0$. In summary, we have the following lemma.

\begin{lemma}\label{LM:diag}
Let~$f\in k[[t, x]]$ and~$F=f(x, t/x)/x\in \mathcal {F}[[t]]$. If~$L\in k[t]\langle D_t\rangle$ is
a telescoper for~$F$
such that~$L(F)=D_x(G)$ with~$G\in \mathcal {F}[[t]]$, then~$L(\operatorname{diag}(f))=0$.
\end{lemma}
%

In the following, we prove Furstenberg's diagonal theorem.
\begin{theorem}[Furstenberg, 1967]\label{THM:diag}
Let~$f\in k[[t, x]]\cap k(t, x)$.
Then the diagonal of~$f$ is a power series algebraic over~$k(t)$.
\end{theorem}
\begin{proof}
Let~$F=f(x, t/x)/x$. Since~$f$ is a rational function in~$k(t, x)$, so is~$F$.
Let~$L\in k(t)\langle D_t \rangle$ be the minimal telescoper for~$F$. Since
multiplying by an element of~$k[t]$ commutes with the derivation~$D_x$, we can always
assume that the coefficients of~$L$ are polynomials in~$k[t]$.
By Theorem~\ref{THM:telecc},
all of the solutions of~$L(y(t))=0$ are algebraic over~$k(t)$. So the diagonal of~$f$
is algebraic over~$k(t)$ since~$L(\operatorname{diag}(f))=0$ by Lemma~\ref{LM:diag}.
\end{proof}

The following example is borrowed from the recent paper by Ekhad and Zeilberger~\cite{EZ2011}, from which
one can see how Zeilberger's method of creative telescoping
plays a role in solving concrete problems in combinatorics.

\begin{example}
Let~$s(n)$ be the number of binary words of length~$n$ for which the number of occurrences of~$00$
is the same as that of~$01$ as subwords. Stanley~\cite{Stanley2011} asked for a proof of
the following formula
\begin{equation}\label{EQ:sf}
S(t) \triangleq \sum_{n=0}^{\infty} s(n) t^n = \frac{1}{2}\left(\frac{1}{1-t} + \frac{1+2t}{\sqrt{(1-t)(1-2t)(1+t+2t^2)}}
\right).
\end{equation}
We first show that the generating function~$S(t)$ is an algebraic function over~$k(t)$.
The key ingredient is the Goulden-Jackson cluster method~\cite{GJ1979}. Noonan and
Zeilberger~\cite{NoonanZeilberger1999} gave an elegant survey of this method
together with an efficient implementation. Let~$\mathcal {W}$ be the set of all binary words
and let~$\tau_{00}(w), \tau_{01}(w)$ be the numbers of occurrences
of~$00$ and~$01$ in~$w\in \mathcal {W}$, respectively. Ekhad and Zeilberger~\cite{EZ2011} define the
generating function
\[f(t, y, z) = \sum_{w\in \mathcal {W}} t^{\text{length(w)}} y^{\tau_{00}(w)}z^{\tau_{01}(w)}.\]
Loading the package~{\sf DAVID{\_}IAN} created by Noonan and Zeilberger to {\sf Maple}, typing
{\sf GJstDetail([0, 1], \{[0, 0], [0, 1]\}, t, s)}, and replacing~$s[0, 0], s[0, 1]$
by~$y, z$, respectively, we get an explicit form of~$f(t, y, z)$,
\[f(t, y, z) = \frac{(1-y)t +1}{(y-z)t^2-(1+y)t+1},\]
which is a rational function of three variables. By definition, the desired generating function~$S(t)$
is the coefficient of~$x^{-1}$ in~$F(t, x) := x^{-1}f(t, x, x^{-1})$.
Since~$\tau_{00}(w)$ and~$\tau_{01}(w)$ are bounded by~${\text{length(w)}}$, the function~$F(t, x)$
is an element in the ring~$k((x))[[t]]$. Therefore, the coefficient~$[x^{-1}](F)$ is annihilated
by any telescoper for~$F$ in~$k[t]\langle D_t\rangle$. By Theorem~\ref{THM:telecc}, the function~$S(t)$
must be an algebraic function over~$k(t)$. By typing~{DETools[Zeilberger](F, t, x, Dt)} in {\sf Maple},
we get the minimal telescoper~$L$ for~$F$, which is
\begin{align*}
L =& \left( -1+5\,t-13\,{t}^{2}-30\,{t}^{4}+23\,{t}^{3}+40\,{t}^{5}-40\,{t}^{6}+16\,{t}^{7}
\right) {{\it Dt}}^{2} \\
&\, \, + \left( 80\,{t}^{6}-168\,{t}^
{5}+152\,{t}^{4}-88\,{t}^{3}+24\,{t}^{2}-2\,t+2 \right) {\it Dt}\\
&\, \, +48\,{
t}^{5}-72\,{t}^{4}+48\,{t}^{3}-12\,{t}^{2}-6\,t.
\end{align*}
To show Stanley's formula~\eqref{EQ:sf}, it suffices to verify that~$S(t)$ satisfies the equation~$L(y(t)) = 0$,
and check the two initial condition:~$y(0) =1$ and~$D_t(y)(0) = 2$.
Moreover, we could also rediscover Stanley's formula by solving the
differential equation. Thanks to Zeilberger's method, many classical combinatorial identities now can be
proved and rediscovered automatically all by computer.
\end{example}

Except the case when~$\partial_t=D_t$ and~$\partial_x=D_x$ as above, we will show that
telescopers  for non-integrable or non-summable rational functions in~$k(t, x)$
have at least one nonzero rational solution in~$k(t)$. Of these 8 cases,
6 follow easily from an examination of some of the proofs above.
These cases are considered in  Theorem~\ref{THM:telemixed}.
The remaining two cases require a slightly more detailed
proof and are considered in Theorem~\ref{THM:teleddqq}.

\begin{theorem}\label{THM:telemixed}
Let~$L\in k(t)\langle \partial_t\rangle$ and~$f\in k(t, x)$ satisfy one of the following conditions:
\begin{enumerate}
  \item $\partial_t = D_t$ and~$f\notin \Delta_x(k(t, x))$;
  \item $\partial_t = D_t$ and~$f\notin \Delta_{q, x}(k(t, x))$;
  \item $\partial_t = S_t$ and~$f\notin D_x(k(t, x))$;
  \item $\partial_t = S_t$ and~$f\notin \Delta_{q, x}(k(t, x))$;
  \item $\partial_t = Q_t$ and~$f\notin D_x(k(t, x))$;
  \item $\partial_t = Q_t$ and~$f\notin \Delta_x(k(t, x))$.
\end{enumerate}
Then~$L(t, \partial_t)$ is a telescoper for some~$f\in k(t, x)$ if and only if~$L(y(t))=0$ has
a nonzero rational solution in~$k(t)$.
\end{theorem}
\begin{proof}
Suppose that~$L(y(t))=0$ has a nonzero rational solution~$r(t)$ in~$k(t)$. Then~$L$ is a telescoper
for~$f=r(t)/x$ and~$f$ satisfies the assumption above. For the opposite implication, Theorems~\ref{THM:cd},
\ref{THM:cq}, \ref{THM:dc}, \ref{THM:dq}, \ref{THM:qc} and~\ref{THM:qd} imply that the residual form of~$f$
is of the form~$a/b$ such that~$b=b_1(t)b_2(x)$ with~$b_1\in k[t]$ and~$b_2\in k[x]$. Then
\[\frac{a}{b} = \sum_{i=1}^m \sum_{j=1}^{n_i} \frac{\alpha_{i, j}}{(x-\beta_i)^j}, \]
where~$\alpha_{i,j}\in k(t)$ and~$\beta_i\in k$ are in distinct ($q$-)orbits.
If~$L$ is a telescoper for~$f$, then~$L$ is also a telescoper for~$a/b$.
Since all the~$\beta_i$ are free of~$t$, we have
\[L(a/b) = \sum_{i=1}^m \sum_{j=1}^{n_i} \frac{L(\alpha_{i, j})}{(x-\beta_i)^j}=\partial_x(g), \quad
\text{where~$\partial_x\in \{D_x, \Delta_x, \Delta_{q, x}\}$}.\]
By Propositions~\ref{PROP:ratint}, \ref{PROP:ratsum}, and~\ref{PROP:ratqsum}, we have~$L(\alpha_{i,j})=0$.
Since~$a/b$ is not zero, at least one of the~$\alpha_{i, j}$ is nonzero. Thus~$L(y(t))=0$
has at least one nonzero rational solution in~$k(t)$.
\end{proof}

\begin{theorem}\label{THM:teleddqq}
Let~$L\in k(t)\langle \partial_t\rangle$ and~$f\in k(t, x)$ satisfy one of the following conditions:
$(1)$~$\partial_t = S_t$ and~$f\notin \Delta_{x}(k(t, x))$; $(2)$~$\partial_t = Q_t$ and~$f\notin \Delta_{q, x}(k(t, x))$.
Then~$L(t, \partial_t)$ is a telescoper for some~$f\in k(t, x)$ if and only if~$L(y(t))=0$ has
a nonzero rational solution in~$k(t)$.
\end{theorem}
\begin{proof}
Suppose that~$L(y(t))=0$ has a nonzero rational solution~$r(t)$ in~$k(t)$. Then~$L$ is a telescoper
for~$f=r(t)/x$ and~$f$ satisfies the assumption above.
For the opposite implication, we only prove the assertion for the first case, that is, when~$L$ and~$f$ satisfies the condition~$(1)$.
The remaining assertion follows in a similar manner.
Theorem~\ref{THM:dd} implies that the residual form~$a/b$ of~$f$ can be decomposed into
\[\frac{a}{b} = \sum_{i=1}^m \sum_{j=1}^{n_i} \frac{\alpha_{i, j}}{(x-\beta_i)^j},\]
where~$\alpha_{i, j}\in k(t)$ and~$\beta_i =\frac{\lambda_i}{\mu_i}t+c_i$ with~$c_i\in k$, $\lambda_i\in \bZ$
and~$\mu_i\in \bN$ such that~$\gcd(\lambda_i, \mu_i)=1$ and the~$\beta_i$ are in distinct~$\bZ$-orbits.
If~$L\in k(t)\langle D_t\rangle$ is a telescoper for~$f$, then~$L$ is a telescoper for~$a/b$.
Moreover, $L$ is a telescoper for each fraction~$f_{i, j} = {\alpha_{i, j}}/{(x-\beta_i)^j}$.
We claim that the
operator~$L_{i, j} := \alpha_{i,j}(t)S_t^{\mu_i} - \alpha_{i, j}(t+\mu_i)\in k(t)\langle D_t \rangle$ is the minimal
telescoper for~$f_{i, j}$ with respect to~$\Delta_x$. In fact, $L_{i, j}$ is a telescoper for~$f_{i,j}$
as shown in the proof of~Theorem~\ref{THM:dd}. It remains to show the minimality. Assume that there exists
a telescoper~$\tilde{L}_{i, j}$ of order less than~$\mu_i$ for~$f_{i, j}$.
Write~$\tilde{L}_{i, j}=\sum_{\ell=0}^{\mu_i-1} e_{\ell}S_t^{\ell}$. Then
\[\tilde{L}_{i, j}(f_{i, j}) = \sum_{\ell=0}^{\mu_i-1} \frac{e_{\ell}
 \alpha_{i, j}(t+\ell)}{(x-(\frac{\lambda_i}{\mu_i}t+\frac{\lambda_i}{\mu_i}\ell + c_i))^j}.\]
Since~$\gcd(\lambda_i, \mu_i)=1$ and~$\ell\in \{0, \ldots, \mu_i-1\}$,
the values~$\frac{\lambda_i}{\mu_i}t+\frac{\lambda_i}{\mu_i}\ell + c_i$ are in distinct~$\bZ$-orbits.
If~$\tilde{L}_{i, j}(f_{i, j})$ is rational summable, then all the residues~$e_{\ell}
 \alpha_{i, j}(t+\ell)$ are zero by Proposition~\ref{PROP:ratsum}. Since~$\alpha_{i,j} \neq 0$, we
have~$\tilde{L}_{i,j}$ is a zero operator. The claim holds. Since~$L$ is a telescoper for~$f_{i, j}$,
$L_{i, j}$ divides~$L$ on the right. Note that the rational function~$\alpha_{i, j}\in k(t)$
is a nonzero solution of~$L_{i, j}(y(t))=0$. Thus, $L$ has at least one nonzero rational solution in~$k(t)$.
\end{proof}


{\section*{Appendix}\label{SECT:appendix}
In this appendix, we present proofs of Propositions~\ref{PROP:afrde} and~\ref{PROP:aflqe}. Let~$K \subset E$ be difference fields of characteristic zero
with automorphism~$\sigma$ and assume that the constants
$E^\sigma$ of~$E$ are in~$K$.  Furthermore assume that~$E$ is
algebraically closed.

\begin{lemma}\label{LM:shiftclose}
Let~$u\in E$ be algebraic over~$K$
and assume that~$u$ satisfies a homogeneous linear difference
equation over~$K$. Then there exists a field~$F \subset E$
with~$\sigma(F) = F$, $K \subset F$,  $[F : K]< \infty$, and~$u \in F$.
\end{lemma}
\begin{proof}
Let $u$ satisfy
\begin{equation}\label{eqn1}
\sigma^n(u) + b_{n-1} \sigma^{n-1}(u) + \cdots + b_0u = 0
\end{equation}
with~$b_i \in K, b_0 \neq 0$ and let~$F = K(u, \sigma(u), \ldots , \sigma^{n-1}(u))$.
We have that~$[F:K] < \infty$ since for any~$i$, $\sigma^i(u)$ is algebraic over~$K$.
To see that~$\sigma(F) \subset F$ it is enough to show that~$\sigma^i(u) \in F$ for all~$i$.
This is certainly true for~$i = 0, \ldots n$. If~$i > n$,  apply~$\sigma^{i-n}$ to
equation~\eqref{eqn1} and proceed by induction to conclude~$\sigma^{i}(u) \in F$.
If~$i <0$ apply~$\sigma^{i}$ and proceed by induction to conclude~$\sigma^i(u) \in F$.
\end{proof}

\begin{lemma}\label{LM:aflre1}
Let~$K = k(t)$, where~$k$ is algebraically closed.
Let~$(E, \sigma)$ be a difference field such that~$K \subset E$, $\sigma(t) = t+1$
and~$[E : K] < \infty$. The~$E = K$.
\end{lemma}
\begin{proof}
Let~$n = [E:K]$ and~$g$ be the genus of~$E$.
The Riemann-Hurwitz formula (see~\cite[p.\ 106]{Chevalley1951} or~\cite[p.\ 125]{Fulton1989})
yields
\begin{equation}\label{RHfmla}
2g-2  =  -2n + \sum_P(e(P) - 1),
\end{equation}
where the sum is over all places~$P$ of~$E$ and~$e(P)$ is the ramification index of~$P$ with
respect to~$K$.
There are only a finite number of places~$Q$ of~$K$ over
which places of~$E$ ramify and the automorphism~$\sigma$ leaves the set of
such places invariant. On the other hand, the only finite set of
places of~$K$ that is left invariant by~$\sigma$ is the place at
infinity.  Therefore, if~$P$ is a place of~$E$ with~$e(P) > 1$,
then~$P$ lies above the place at infinity. Note that for any place~$Q$
of~$K$, Theorem~1 of~\cite[p.\ 52]{Chevalley1951} implies (under our
assumptions) that
\begin{equation}\label{ramsum}
\sum_{\mbox{$P$ lies above~$Q$}} e(P) = n.
\end{equation}
Therefore we have
\begin{eqnarray*}
2g-2 & = & -2n + \sum_{\mbox{$P$ lies above~$\infty$}}(e(P) - 1)\\& =&
-2n +n-t \\
& =& -n-t,
\end{eqnarray*}
where~$t$ is the number of
places above infinity. Since~$n$ and~$t$ are both positive integers
and~$g$ is nonnegative, we must have~$g=0$ and~$n = t = 1$.
In particular, since~$n = 1$, we have~$E= K$.
\end{proof}

\noindent{\underline{\emph{Proof of Proposition~\ref{PROP:afrde}.}}
Suppose that~$\alpha(t)$ satisfies the linear recurrence relation
\[S_t^n(\alpha) + a_{n-1} S_t^{n-1} (\alpha) + \cdots + a_0 \alpha = 0,\]
where~$a_i\in k(t)$. By Lemma~\ref{LM:shiftclose}, the
field~$E = k(t)(\alpha, S_t(\alpha), \ldots, S_t^{n-1}(\alpha)) \subset \overline{k(t)}$
is a difference field extension of~$k(t)$. Since~$[E:k(t)]<\infty$, $E=k(t)$ by Proposition~\ref{LM:aflre1}.
Thus~$\alpha\in k(t)$. \hfill $\Box$

\begin{remark} Proposition~\ref{PROP:afrde} has been shown in~\cite[Theorem 1]{Benzaghou1992},~\cite[Prop.\ 4.4]{vdPutSinger1997}
and~\cite[Theorem 5.2]{Bell2008}. The proof in~\cite[Theorem 5.2]{Bell2008} is based on
analytic properties of algebraic functions.\\[0.1in]
In this proposition, we assume that~$\alpha(t)$ satisfies a polynomial equation over~$k(t)$ and \underline{lies in a field}. This latter condition cannot be weakened without weakening the conclusion. For example, the sequence~$y=(-1)^n$ satisfies $y^2-1 = 0$ but~$k(t)[y]$ is a ring with zero divisors.  The above references give a complete characterization of sequences satisfying both linear recurrences and polynomial equations.
\end{remark}

The following result is a~$q$-analogue of Lemma~\ref{LM:aflre1}.
\begin{lemma}\label{LM:aflqre1}
Let~$K = k(t)$, where~$k$ is algebraically closed. Let~$(E, \sigma)$ be
a difference field such that~$K \subset E$, $\sigma(t) = qt$ with~$q \in k\setminus\{0\}$ and
not a root of unity, and~$[E:K] < \infty$. Then~$E = k(t^{{1}/{n}})$ for some
positive integer~$n$.
\end{lemma}
\begin{proof}
Let~$[E:K] = n$ and~$g$ be the genus of~$E$. We again
consider the set of places of~$K$ over which places of~$E$ ramify.
This set is left invariant by~$\sigma$ and so must be a subset of
the set containing the place at~$0$ and the place at~$\infty$.
Therefore, ramification can occur only at~$0$ and~$\infty$.
Equations~\eqref{RHfmla} and~\eqref{ramsum} imply
\begin{eqnarray*}
2g-2 & = & -2n + \sum_{\mbox{~$P$ lies above~$0$}}(e(P) - 1) +
\sum_{\mbox{$P$ lies above~$\infty$}}(e(P) - 1)\\
 &=&-2n +2n-t_0-t_\infty  \\ &=& -t_0-t_\infty
\end{eqnarray*}
where~$t_0, t_\infty$ are the number of places above~$0$ and~$\infty$.
Since~$t_0$ and~$t_\infty$ are positive and~$g$ is
nonnegative, we must have that~$g=0$ and~$t_0 = t_\infty = 1$.
Therefore, $E$ has one place~$P_0$ over~$0$ with~$e_{P_0} = n$ and
one place~$P_\infty$ over~$\infty$ with~$e_{P_\infty} = n$. Writing
divisors multiplicatively, Riemann's Theorem (\cite[p.\ 22]{Chevalley1951}) implies that
\begin{equation*}
l(P_0P^{-1}_\infty)  \geq  d(P_0P^{-1}_\infty) -g+1 = 0-0+1 = 1
\end{equation*}
where~$l(P_0P^{-1}_\infty)$ is the dimension of the space of
elements of~$E$ which are~$\equiv 0 \mod {P_0P^{-1}_\infty}$.
Note that since the degree~$P_0P^{-1}_\infty$ is~$0$, this latter condition
implies that any such element has~$P_0P^{-1}_\infty$ as its divisor.
Therefore, there exists an element~$y \in E$ whose divisor is~$P_0P^{-1}_\infty$.
Note that the element~$t$ has divisor~$P_0^nP^{-n}_\infty$ and therefore~$y^nt^{-1}$
must be in~$k$. Therefore~$y = c\, t^{{1}/{n}}$ for some~$c \in k$.  Finally,
Theorem 4 of~\cite[p.\ 18]{Chevalley1951} states that~$[E:k(y)]$
equals the degree of the divisor of zeros of~$y$, that is,
$[E:k(y)] = 1$.  Therefore~$ E = k(y) = k(t^{{1}/{n}})$.
\end{proof}

\noindent{\underline{\emph{Proof of Proposition~\ref{PROP:aflqe}.}}
Suppose that~$\alpha(t)$ satisfies the linear $q$-recurrence relation
\[Q_t^n(\alpha) + a_{n-1} Q_t^{n-1} (\alpha) + \cdots + a_0 \alpha = 0,\]
where~$a_i\in k(t)$. By Lemma~\ref{LM:shiftclose}, the
field~$E = k(t)(\alpha, Q_t(\alpha), \ldots, Q_t^{n-1}(\alpha)) \subset \overline{k(t)}$
is a difference field extension of~$k(t)$. Since~$[E:k(t)]<\infty$, $E=k(t^{1/n})$
by Lemma~\ref{LM:aflqre1}.
Thus~$\alpha\in k(t^{1/n})$.} \hfill $\Box$


\bibliographystyle{plain}

\begin{thebibliography}{10}

\bibitem{Abel1881}
N.~H. Abel.
\newblock {\em {\OE}uvres {C}ompl\`etes de {N}iels {H}enrik {A}bel. {T}ome
  {II}}.
\newblock Imprimerie de Gr\o ndahl \& Son, Christiania, 1981.
\newblock Contenant les m{\'e}moirs posthumes d'Abel. [Containing the
  posthumous memoirs of Abel], Edited and with notes by L. Sylow and S. Lie.

\bibitem{Abramov1975}
S.~A. Abramov.
\newblock The rational component of the solution of a first order linear
  recurrence relation with rational right hand side.
\newblock {\em \v Z. Vy\v cisl. Mat. i Mat. Fiz.}, 15(4):1035--1039, 1090,
  1975.

\bibitem{Abramov1989}
S.~A. Abramov.
\newblock Rational solutions of linear differential and difference equations
  with polynomial coefficients.
\newblock {\em USSR Comput. Math. Math. Phys.}, 29(6):7--12, 1989.

\bibitem{Abramov1995b}
S.~A. Abramov.
\newblock Indefinite sums of rational functions.
\newblock In {\em ISSAC '95: Proceedings of the 1995 {I}nternational
  {S}ymposium on {S}ymbolic and {A}lgebraic {C}omputation}, pages 303--308, New
  York, NY, USA, 1995. ACM.

\bibitem{Abramov1995}
S.~A. Abramov.
\newblock Rational solutions of linear difference and $q$-difference equations
  with polynomial coefficients.
\newblock In {\em ISSAC '95: Proceedings of the 1995 {I}nternational
  {S}ymposium on {S}ymbolic and {A}lgebraic {C}omputation}, pages 285--289, New
  York, NY, USA, 1995. ACM.

\bibitem{Abramov2003}
S.~A. Abramov.
\newblock When does {Z}eilberger's algorithm succeed?
\newblock {\em Adv. in Appl. Math.}, 30(3):424--441, 2003.

\bibitem{AbramovYu1991}
S.~A. Abramov and K.~Y. Kvansenko.
\newblock Fast algorithms to search for the rational solutions of linear
  differential equations with polynomial coefficients.
\newblock In {\em ISSAC '91: Proceedings of the 1991 {I}nternational
  {S}ymposium on {S}ymbolic and {A}lgebraic {C}omputation}, pages 267--270, New
  York, NY, USA, 1991. ACM.

\bibitem{AbramovLe2002}
S.~A. Abramov and H.~Q. Le.
\newblock A criterion for the applicability of {Z}eilberger's algorithm to
  rational functions.
\newblock {\em Discrete Math.}, 259(1-3):1--17, 2002.

\bibitem{Almkvist1990}
G.~Almkvist and D.~Zeilberger.
\newblock The method of differentiating under the integral sign.
\newblock {\em J. Symbolic Comput.}, 10:571--591, 1990.

\bibitem{Marshall2005}
J.~M. Ash and S.~Catoiu.
\newblock Telescoping, rational-valued series, and zeta functions.
\newblock {\em Trans. Amer. Math. Soc.}, 357(8):3339--3358 (electronic), 2005.

\bibitem{Bell2008}
J.~P. Bell, S.~Gerhold, M.~Klazar, and F.~Luca.
\newblock Non-holonomicity of sequences defined via elementary functions.
\newblock {\em Ann. Comb.}, 12(1):1--16, 2008.

\bibitem{Benzaghou1992}
B.~Benzaghou and J.-P. B{\'e}zivin.
\newblock Propri\'et\'es alg\'ebriques de suites diff\'erentiellement finies.
\newblock {\em Bull. Soc. Math. France}, 120(3):327--346, 1992.

\bibitem{BCLSS2007}
A.~Bostan, F.~Chyzak, B.~Salvy, G.~Lecerf, and {\'E}.~Schost.
\newblock Differential equations for algebraic functions.
\newblock In {\em ISSAC'07: Proceedings of the 2007 {I}nternational {S}ymposium
  on {S}ymbolic and {A}lgebraic {C}omputation}, pages 25--32. ACM, New York,
  2007.

\bibitem{BronsteinBook}
M.~Bronstein.
\newblock {\em {S}ymbolic {I}ntegration {I}: {T}ranscendental {F}unctions},
  volume~1 of {\em Algorithms and Computation in Mathematics}.
\newblock Springer-Verlag, Berlin, second edition, 2005.

\bibitem{CCFL2010}
S.~Chen, F.~Chyzak, R.~Feng, and Z.~Li.
\newblock The existence of telescopers for hyperexponential-hypergeometric
  functions, 2010.
\newblock MM-Res. Preprints (2010) No. 29, 239-267.

\bibitem{CKS2012}
S.~Chen, M.~Kauers, and M.~F. Singer.
\newblock Telescopers for rational and algebraic functions via residues, 2012.
\newblock {http://arxiv.org/abs/1201.1954}.

\bibitem{Chen2005}
W.~Y.~C. Chen, Q.-H. Hou, and Y.-P. Mu.
\newblock Applicability of the {$q$}-analogue of {Z}eilberger's algorithm.
\newblock {\em J. Symbolic Comput.}, 39(2):155--170, 2005.

\bibitem{Chevalley1951}
C.~Chevalley.
\newblock {\em Introduction to the {T}heory of {A}lgebraic {F}unctions of {O}ne
  {V}ariable}.
\newblock Mathematical Surveys, No. VI. American Mathematical Society, New
  York, N. Y., 1951.

\bibitem{Chudnovsky1986}
D.~V. Chudnovsky and G.~V. Chudnovsky.
\newblock On expansion of algebraic functions in power and {P}uiseux series.
  {I}.
\newblock {\em J. Complexity}, 2(4):271--294, 1986.

\bibitem{Chudnovsky1987}
D.~V. Chudnovsky and G.~V. Chudnovsky.
\newblock On expansion of algebraic functions in power and {P}uiseux series.
  {II}.
\newblock {\em J. Complexity}, 3(1):1--25, 1987.

\bibitem{Cockle1861}
J.~Cockle.
\newblock On transcendental and algebraic solution.
\newblock {\em Philosophical Magazine}, XXI:379--383, 1861.

\bibitem{Comtet1964}
L.~Comtet.
\newblock Calcul pratique des coefficients de {T}aylor d'une fonction
  alg\'ebrique.
\newblock {\em Enseignement Math. (2)}, 10:267--270, 1964.

\bibitem{CormierSingerTragerUlmer2002}
O.~Cormier, M.~F. Singer, B.~M. Trager, and F.~Ulmer.
\newblock Linear differential operators for polynomial equations.
\newblock {\em J. Symbolic Comput.}, 34(5):355--398, 2002.

\bibitem{DH2011}
L.~Di~Vizio and C.~Hardouin.
\newblock {Descent for differential Galois theory of difference equations
  Confluence and q-dependency}, 2011.
\newblock {http://arxiv.org/abs/1103.5067}.

\bibitem{EZ2011}
S.~B. Ekhad and D.~Zeilberger.
\newblock {Automatic Solution of Richard Stanley's Amer. Math. Monthly Problem
  \#11610 and {A}{N}{Y} Problem of That Type}, 2011.
\newblock {http://arxiv.org/abs/1112.6207v1}.

\bibitem{Flajolet_Sedgewick}
P.~Flajolet and R.~Sedgewick.
\newblock {\em Analytic combinatorics}.
\newblock Cambridge University Press, Cambridge, 2009.

\bibitem{Fliess1974}
M.~Fliess.
\newblock Sur divers produits de s\'eries formelles.
\newblock {\em Bull. Soc. Math. France}, 102:181--191, 1974.

\bibitem{Fulton1989}
W.~Fulton.
\newblock {\em Algebraic curves}.
\newblock Advanced Book Classics. Addison-Wesley Publishing Company Advanced
  Book Program, Redwood City, CA, 1989.
\newblock An introduction to algebraic geometry, Notes written with the
  collaboration of Richard Weiss, Reprint of 1969 original.

\bibitem{Furstenberg1967}
H.~Furstenberg.
\newblock Algebraic functions over finite fields.
\newblock {\em J. Algebra}, 7:271--277, 1967.

\bibitem{Gessel1980}
I.~M. Gessel.
\newblock A factorization for formal {L}aurent series and lattice path
  enumeration.
\newblock {\em J. Combin. Theory Ser. A}, 28(3):321--337, 1980.

\bibitem{GJ1979}
I.~P. Goulden and D.~M. Jackson.
\newblock An inversion theorem for cluster decompositions of sequences with
  distinguished subsequences.
\newblock {\em J. London Math. Soc. (2)}, 20(3):567--576, 1979.

\bibitem{Haiman1993}
M.~Haiman.
\newblock Noncommutative rational power series and algebraic generating
  functions.
\newblock {\em European J. Combin.}, 14(4):335--339, 1993.

\bibitem{Hardouin2008}
C.~Hardouin and M.~F. Singer.
\newblock Differential {G}alois theory of linear difference equations.
\newblock {\em Math. Ann.}, 342(2):333--377, 2008.

\bibitem{Harley1862}
R.~Harley.
\newblock On the theory of the transcendental solution of algebraic equations.
\newblock {\em Quart. J. of Pure and Applied Math}, 5:337--361, 1862.

\bibitem{Hermite1872}
C.~Hermite.
\newblock Sur l'int\'egration des fractions rationnelles.
\newblock {\em Ann. Sci. \'Ecole Norm. Sup. (2)}, 1:215--218, 1872.

\bibitem{Horowitz1971}
E.~Horowitz.
\newblock Algorithms for partial fraction decomposition and rational function
  integration.
\newblock In {\em SYMSAC'71}, pages 441--457, New York, USA, 1971. ACM.

\bibitem{Le2001}
H.~Q. Le.
\newblock On the $q$-analogue of {Z}eilberger's algorithm to rational
  functions.
\newblock {\em Program. Comput. Softw.}, 27(1):35--42, 2001.

\bibitem{Matusevich2000}
L.~F. Matusevich.
\newblock Rational summation of rational functions.
\newblock {\em Beitr\"age Algebra Geom.}, 41(2):531--536, 2000.

\bibitem{Nahay2003}
J.~M. Nahay.
\newblock Linear relations among algebraic solutions of differential equations.
\newblock {\em J. Differential Equations}, 191(2):323--347, 2003.

\bibitem{NoonanZeilberger1999}
J.~Noonan and D.~Zeilberger.
\newblock The {G}oulden-{J}ackson cluster method: extensions, applications and
  implementations.
\newblock {\em J. Differ. Equations Appl.}, 5(4-5):355--377, 1999.

\bibitem{Ostrogradsky1845}
M.~V. Ostrogradski{\u\i}.
\newblock De l'int{\'e}gration des fractions rationnelles.
\newblock {\em Bull.\ de la classe physico-math{\'e}matique de l'Acad.\
  Imp{\'e}riale des Sciences de Saint-P{\'e}tersbourg}, 4:145--167, 286--300,
  1845.

\bibitem{Paule1995b}
P.~Paule.
\newblock Greatest factorial factorization and symbolic summation.
\newblock {\em J. Symbolic Comput.}, 20(3):235--268, 1995.

\bibitem{Pirastu1995a}
R.~Pirastu.
\newblock Algorithms for indefinite summation of rational functions in maple.
\newblock In {\em Maple. MapleTech}, pages 29--38, 1995.

\bibitem{Pirastu1995b}
R.~Pirastu and V.~Strehl.
\newblock Rational summation and {G}osper-{P}etkov\v sek representation.
\newblock {\em J. Symbolic Comput.}, 20(5-6):617--635, 1995.
\newblock Symbolic computation in combinatorics (Ithaca, NY, 1993).

\bibitem{Rothstein1977}
M.~Rothstein.
\newblock A new algorithm for integration of exponential and logarithmic
  functions.
\newblock In {\em Proceedings of the 1977 MACSYMA Users Conference (Berkeley,
  CA)}, pages 263--274, Washington DC, 1977. NASA.

\bibitem{Schneider2010}
C.~Schneider.
\newblock Parameterized telescoping proves algebraic independence of sums.
\newblock {\em Ann. Comb.}, 14(4):533--552, 2010.

\bibitem{Stanley2011}
R.~Stanley.
\newblock Problems and {S}olutions: {P}roblems: 11610.
\newblock {\em Amer. Math. Monthly}, 118(10):937, 2011.

\bibitem{Stanley1999}
R.~P. Stanley.
\newblock {\em Enumerative {C}ombinatorics. {V}ol. 2}, volume~62 of {\em
  Cambridge Studies in Advanced Mathematics}.
\newblock Cambridge University Press, 1999.

\bibitem{Trager1976}
B.~M. Trager.
\newblock Algebraic factoring and rational function integration.
\newblock In {\em SYMSAC'76: Proceedings of the Third ACM Symposium on Symbolic
  and Algebraic Computation}, pages 219--226, New York, NY, USA, 1976. ACM.

\bibitem{VanDerPut2001}
M.~van~der Put.
\newblock Grothendieck's conjecture for the {R}isch equation {$y'=ay+b$}.
\newblock {\em Indag. Math. (N.S.)}, 12(1):113--124, 2001.

\bibitem{vdPutSinger1997}
M.~van~der Put and M.~F. Singer.
\newblock {\em Galois {T}heory of {D}ifference {E}quations}, volume 1666 of
  {\em Lecture Notes in Mathematics}.
\newblock Springer-Verlag, Berlin, 1997.

\bibitem{WilfZeilberger1990b}
H.~S. Wilf and D.~Zeilberger.
\newblock Rational functions certify combinatorial identities.
\newblock {\em J. Amer. Math. Soc.}, 3(1):147--158, 1990.

\bibitem{WilfZeilberger1990a}
H.~S. Wilf and D.~Zeilberger.
\newblock Towards computerized proofs of identities.
\newblock {\em Bull. Amer. Math. Soc. (N.S.)}, 23(1):77--83, 1990.

\bibitem{Wilf1992}
H.~S. Wilf and D.~Zeilberger.
\newblock An algorithmic proof theory for hypergeometric (ordinary and
  ``{$q$}'') multisum/integral identities.
\newblock {\em Invent. Math.}, 108(3):575--633, 1992.

\bibitem{Zeilberger1990}
D.~Zeilberger.
\newblock A holonomic systems approach to special functions identities.
\newblock {\em J. Comput. Appl. Math.}, 32:321--368, 1990.

\bibitem{Zeilberger1991}
D.~Zeilberger.
\newblock The method of creative telescoping.
\newblock {\em J. Symbolic Comput.}, 11(3):195--204, 1991.

\end{thebibliography}

\end{document}